\documentclass[a4paper,11pt,reqno]{amsart}
\usepackage[left=26mm,right=26mm,top=30mm,bottom=30mm]{geometry}

\usepackage{amssymb}
\usepackage[shortlabels]{enumitem}
\usepackage{bm}
\usepackage{mathtools}
\usepackage{mathrsfs}
\usepackage{array} 
\usepackage{xcolor} 
\usepackage{xparse}
\usepackage[sort&compress,numbers]{natbib}
\usepackage{caption}
\usepackage{subcaption}
\usepackage{booktabs}
\usepackage{hyperref}
\usepackage{bbm}
\def\1{\mathbbm{1}}

\def\A{{\mathbb{A}}}

\def\ka{\kappa}

\def\la{\lambda}

\def\t{\tau}

\def\calB{{\calB}}

\def\calT{{\mathcal{T}}}

\def\calB{{\mathcal{B}}}

\def\calA{{\mathcal{A}}}

\def\R{\mathbf{R}}

\def\C{\mathbf{C}}

\def\F{\mathbb F}

\def\sub{\subseteq}

\newtheorem{theorem}{Theorem}[section]
\theoremstyle{plain}
\newtheorem{definition}{Definition}[section]
\newtheorem{lemma}{Lemma}[section]
\newtheorem{proposition}{Proposition}[section]
\newtheorem{corollary}{Corollary}[section]
\newtheorem{remark}{Remark}[section]
\numberwithin{equation}{section}

\newtheorem{mainassumptions}{Assumption}[section]

\makeatletter
\newcommand\makebig[2]{%
\@xp\newcommand\@xp*\csname#1\endcsname{\bBigg@{#2}}%
\@xp\newcommand\@xp*\csname#1l\endcsname{\@xp\mathopen\csname#1\endcsname}%
\@xp\newcommand\@xp*\csname#1r\endcsname{\@xp\mathclose\csname#1\endcsname}%
}
\makeatother

\makebig{biggg} {3.0}
\makebig{Biggg} {3.5}
\makebig{bigggg}{4.0}
\makebig{Bigggg}{4.5}
\makebig{biggggg}{5}
\makebig{Biggggg}{5.5}
\makebig{bigggggg}{6}
\makebig{Bigggggg}{6.5}

\makeatletter
\newcommand{\doublehat}[1]{%
\begingroup%
\let\macc@kerna\z@%
\let\macc@kernb\z@%
\let\macc@nucleus\@empty%
\hat{\mathchoice%
{\raisebox{.2ex}{\vphantom{\ensuremath{\displaystyle #1}}}}%
{\raisebox{.2ex}{\vphantom{\ensuremath{\textstyle #1}}}}%
{\raisebox{.16ex}{\vphantom{\ensuremath{\scriptstyle #1}}}}%
{\raisebox{.14ex}{\vphantom{\ensuremath{\scriptscriptstyle #1}}}}%
\smash{\hat{#1}}}%
\endgroup%
}
\makeatother

\NewCommandCopy{\ordinaryexists}{\exists}
\RenewDocumentCommand{\exists}{}{\mathop{{}\ordinaryexists}}

\usepackage{tikz}
\usetikzlibrary{plotmarks,positioning,arrows}

\DeclareRobustCommand\full {\tikz[baseline=-0.6ex]\draw[very thick] (0,0)--(0.5,0);}

\tikzset{%
 block/.style={draw, fill=white, rectangle, 
 minimum height=2em, minimum width=3em},
 input/.style={inner sep=0pt}, 
 output/.style={inner sep=0pt}, 
 sum/.style = {draw, fill=white, circle, minimum size=2mm, node distance=1.5cm, inner sep=0pt},
 pinstyle/.style = {pin edge={to-,thin,black}}
}
\begin{document}
\title{Exponential stability of abstract boundary-coupled positive systems}
\author{Yassine El Gantouh$^{\tt1,\ast}$}
\author{Mahyar Mahinzaeim$^{\tt2}$}

\thanks{
\vspace{-1em}\newline\noindent
\newline\noindent
$^{\tt1}$ School of Mathematics, Southwest Jiaotong University, Chengdu, 611756, Sichuan, China.
\newline\noindent
$^{\tt2}$ Research Center for Complex Systems, Aalen University, Germany.
\newline\noindent
{\sc Emails}:
{\tt elgantouhyassine@gmail.com},~{\tt m.mahinzaeim@web.de}.
\newline\noindent
$^{\ast}$ Corresponding author.
}

\begin{abstract}
In this paper, we study the well-posedness and exponential stability of a class of abstract boundary-coupled positive systems. Within a general semigroup framework, we establish well-posedness in terms of existence, uniqueness, and regularity of solutions. More concretely, we derive some simple and readily verifiable tests on the boundary coupling operators for the verification of the well-posedness of the coupled system and, at the same time, provide spectral criteria that guarantee exponential stability of the solutions. The analysis is based on two complementary approaches: the feedback-theoretic method and perturbation techniques for resolvent positive operators on Banach lattices (i.e., Banach spaces equipped with a compatible order structure). This unified framework allows for a systematic treatment of a broad class of boundary couplings in abstract dynamical systems. Several examples are presented to illustrate the applicability of the theoretical results.
\end{abstract}
\maketitle

\pagestyle{myheadings} \thispagestyle{plain} \markboth{\sc Y.\ EL Gantouh, M.\ Mahinzaeim}{\sc abstract boundary-coupled positive systems}

\section{Introduction}

We consider in this paper a class of dynamical systems which can be described abstractly by
\begin{equation}\label{dynamic-boundary}
\dot{z}_1(t)={{A}}_1 z_1(t),\quad \dot{z}_2(t)={{A}}_2 z_2(t),\quad t>0,
\end{equation}
with boundary and initial conditions
\begin{gather}\label{dynamic-boundary2}
G_1 z_1(t)= K_2 z_2(t),\quad G_2 z_2(t)=K_1 z_1(t),\quad t>0,\quad z_1(0)=x_1,\quad z_2(0)=x_2;
\end{gather}
here and in what follows dots denote (partial) differentiation with respect to time $t$ as usual. For $k=1,2$, we let $X_k$ and $U_k$ be infinite-dimensional complete normed vector spaces (or Banach spaces) over the same field -- real or complex -- and we let the initial points $x_k\in X_k$, the mappings $A_k:\bm{D}({{A}}_k)\subset X_k\to X_k$ be closed generally unbounded linear operators with domains of definition $\bm{D}({A}_k)$ dense in $X_k$, and $ G_1:\bm{D}({A}_1)\to U_1$, $K_1: \bm{D}({A}_1)\to U_2$, $G_2:\bm{D}({A}_2)\to U_2$, $K_2: \bm{D}({A}_2)\to U_1$ be linear boundary operators which express the coupling between the two subsystems in \eqref{dynamic-boundary}. In this way \eqref{dynamic-boundary} and \eqref{dynamic-boundary2} together constitute an abstract boundary-coupled system for which we use the acronym $\mathbf{aBCS}$ hereafter. Of particular interest in this study are the well-posedness and stability -- asymptotic or long-time behavior -- analysis of the (solutions of the) $\mathbf{aBCS}$, and we are very much interested in how the coupling between the two subsystems affects the exponential stability analysis (to be defined later). The boundary conditions $G_1 z_1(t)= K_2 z_2(t)$ and $G_2 z_2(t)=K_1 z_1(t)$ above will frequently be referred to in what follows in this paper as the \textit{coupling conditions}.

Many coupled or networked problems that arise in applied mathematics and mathematical physics can be represented by abstract dynamical systems in the form of the $\mathbf{aBCS}$. These include transport equations with delay describing population dynamics of the types considered in \cite{MR4466961,MR3343578}, abstract boundary delay equations introduced in \cite{GANTOUH2026106307}, elliptic and parabolic partial differential equations of possibly quasilinear type \cite{MR1233197,MR1095411}, coupled parabolic and hyperbolic equations describing heat transfer in fluid-solid systems including the semilinear and quasilinear types in \cite{MR1025453,MR4876903,MR4115854}, and linear hyperbolic partial differential equations of possibly higher order describing the wave-like dynamics of elastic systems \cite{MR2105338,MR2046319,MR2421329}. In most of these papers the boundary conditions and couplings are dynamic in nature, by which is meant they contain time derivatives representing damping or inertia coupling, and a number of different assumptions are made concerning these. For example, in \cite{MR3343578}, the well-posedness and exponential stability of the $\mathbf{aBCS}$ when $K_2$ is the identity on $X_2$ were studied. Similar results were also derived in \cite{MR4466961,Boujijane2024} for the case $K_2$ unbounded and $K_1$ bounded. Most recently in \cite{MR4876903}, the authors obtained results on the well-posedness and strong stability of the $\mathbf{aBCS}$ including a certain subclass of passive systems in Hilbert spaces. In particular, they derived explicit resolvent bounds and decay constants for the $\mathbf{aBCS}$. It is the results of \cite{MR4876903} that provide the motivation for the investigation in the paper at hand and that we wish to extend to the case of exponential stability of positive systems in Banach lattices. 

Many systems arising in physics, biology, chemistry, and economics exhibit the natural property that solutions corresponding to positive initial conditions remain positive for all time. Moreover, the concept of positivity in infinite-dimensional systems which are linear and time-invariant makes it possible to encompass coupled physical systems whose semigroup representations imply well-posedness and satisfy a positivity condition which is strong enough to guarantee not only that solutions are always positive, but also, more importantly, that the generated semigroup satisfies the spectrum-determined growth condition (to be defined later). Note that all positive semigroups on spaces of continuous functions and on $\bm{L}_p$-spaces for $p \in [1,\infty)$ have this latter property. The most general results in this regard, obtained by Banach space methods, were established in \cite{MR839450,MR1273529,MR1469440}. For such systems, we can therefore derive a range of interesting exponential stability results. A recent overview of developments in this area is given in \cite{MR4841783} (see also the earlier survey \cite{MR1060534}). Other contributions to the theory of linear time-invariant infinite-dimensional positive systems can be found in \cite{MR4910475,MR3616245,Alessio2026,EZZL,elgantouh2025admissibility,MR4635044,GANTOUH2026106307,Schanbacher1989,MR4711370,Wint}.

The first task in the analysis of the $\mathbf{aBCS}$ is to establish well-posedness via operator-theoretic methods in terms of strongly continuous semigroups of bounded linear operators -- $C_0$-semigroups in short -- on Banach spaces. We propose two complementary perturbation-based approaches to address this issue in the present paper. The first approach is rooted in the feedback theory of positive linear time-invariant infinite-dimensional systems that are regular in the sense of \cite{MR2154892,MR1359020}. Within this framework, we show that the $\mathbf{aBCS}$, when formulated as an abstract initial/boundary-value problem on a suitable Banach state space, has the required properties of existence, uniqueness, and regularity of its solutions. These results are established using a modified version of the Weiss–Staffans perturbation theorem recently obtained in \cite{GANTOUH2026106307}. We derive an explicit convolution formula on the space considered, which enables the definition of solutions to the associated inhomogeneous problem. In addition, we discuss the asymptotic behavior of the $\mathbf{aBCS}$, and we prove that the resulting positive $C_0$-semigroup is exponentially stable if and only if each coupled subsystem is exponentially stable and the transfer function of the overall system has a spectral radius strictly less than one. The second approach is more direct and relies solely on generation and stability results for domain perturbations of positive semigroup generators \cite{MR4711370}. In particular, we show that we can recover the same well-posedness and stability conclusions as those obtained via the first method, provided each of the decoupled subsystems satisfies suitable resolvent-inverse estimates on the positive semiaxis. For related results on the generation under domain perturbations of positive semigroup generators, we refer to \cite{Barbieri,Alessio2026,MR4386712,MR4711370,GANTOUH2026106307,Gantouh2026,PiMa,MR4635044}.

We note that our results will be proven in the context of positive mild solutions under relatively weak assumptions on the system dynamics when compared with other approaches, e.g.\ see \cite{MR3343578,GANTOUH2026106307,MR4466961}. Additionally, we obtain spectral criteria for well-posedness and exponential stability. Hence, the theory of this study is complete in the sense that it addresses the basic questions concerning well-posedness and exponential stability which arise in coupled physical systems and which are important in applications. Indeed, the reader will see that a number of application examples can be made of the theory with little extra effort when compared with the efforts usually required to obtain the well-posedness and exponential stability results (see \cite{MR3624395,Boujijane2024,GANTOUH2026106307,MR3343578,MR4466961}).

\subsection{Contributions and organization}

The contributions of this paper can be summarized as follows. We present a perturbation-based operator-theoretic framework for establishing well-posedness and exponential stability results for the $\mathbf{aBCS}$. While the paper \cite{MR4876903} focused on operator-theoretic formulations in Hilbert space, leading to results for contraction $C_0$-semigroups, we treat the problem in the more general setting of Banach spaces. This generalization is essential for studying diffuse boundary couplings in the $\mathbf{aBCS}$, and thus directly links to the underlying physical situation. Our approach hinges on the inherent positivity of the system dynamics, which serves a dual purpose: first, it relaxes the regularity requirements on the internal dynamics of each subsystem; second, it allows us to exploit the powerful theory of positive semigroups on Banach lattices. The resulting well-posedness and exponential stability results are ``sharp'', meaning that they cannot be relaxed under the imposed assumptions on the coupled system, thereby extending prior findings in \cite{MR4711370,MR4466961,GANTOUH2026106307,MR4876903,MR3343578}. A key contribution of this paper is a spectral characterization for the well-posedness under verifiable conditions on the boundary coupling operators. In particular, we establish simple and readily verifiable tests for $C_0$-semigroup generation, i.e.\ for the well-posedness of the $\mathbf{aBCS}$ (Proposition \ref{sharp-wellposedness-cond}, Theorem \ref{Result1}, and also Remarks \ref{test-remark1} and \ref{test-remark2}). This characterization is derived from a complete description of the spectrum of the positive semigroup generator of the $\mathbf{aBCS}$ (Proposition \ref{Spectrum} and Remark \ref{rem.chara}). As a consequence, considerable emphasis is placed on application examples illustrating the theory.

The remainder of the paper is organized as follows. Some notation, definitions, and auxiliary results which will be used in our considerations are recalled in Section \ref{Sec:1}. Also introduced in that section is the abstract initial/boundary-value problem, which is used as a prototype for describing the $\mathbf{aBCS}$, together with a brief review of the regular systems approach to $C_0$-semigroup generation. In Section \ref{Sec:2}, results are established concerning the well-posedness and exponential stability of the $\mathbf{aBCS}$ in the Banach lattice setting using two complementary approaches: the feedback-theoretic method and perturbation techniques for resolvent positive operators on Banach lattices. Section \ref{Sec:3} presents several applications and physical system models which can be included within our framework. Finally, in Section \ref{Sec:4}, we draw some conclusions.

\section{Notation, terminology and preliminaries}\label{Sec:1}

In this section we introduce standard notation and terminology along with definitions and auxiliary results for the reader's convenience and use in later sections. Further details can be found (in a number of different forms) in the texts \cite{MR2262133,MR423039,MR3616245,MR2154892,MR2502023,MR839450}, or in the papers \cite{MR4910475,MR1359020,GANTOUH2026106307,StaffansWeiss2002}.

\subsection{Notation and terminology}

In this paper, $\R$ and $\R_+$ denote the sets of real and nonnegative real numbers, and $\C$ is the set of complex numbers. Let $E$ and $F$ denote Banach spaces. We denote by $\calB(E,F)$ the space of bounded linear operators from $E$ to $F$ equipped with the strong operator topology; if $E=F$, we write $\calB(E)\coloneqq \calB(E,E)$. For an operator $P\in \calB(E,F)$ and a subspace $Z$ in $E$, we denote by $P\vert_{Z}$ the restriction of $P$ to $Z$, which is again a bounded operator. The notation $\operatorname{ran}P$ is used to denote the range of $P$. For any (real) $p\in [1,\infty)$, we use the notation $\bm{L}_p(\R_+;E)$ to denote the space of all strongly measurable, $p$-integrable $E$-valued functions. For each compact interval $[a,b]\subset \R$, we denote by $\bm{C}(a,b;E)$ the space of all continuous $E$-valued functions on $[a,b]$, and we denote by $\bm{W}^{1}_p(a,b;E)$ the Sobolev space of absolutely continuous $E$-valued functions on $[a,b]$ whose first-order derivatives lie in $\bm{L}_p(a,b;E)$.

Let $E\coloneqq (E,\le)$ be a Banach lattice, i.e.\ a partially ordered Banach space for which any given elements $x$, $y$ of $E$ have a supremum denoted by $x \vee y$ with respect to the partial ordering $\le$ and we denote by $\vert x\vert$ the absolute value defined by $\vert x\vert \coloneqq x\vee (-x)$ if $x\in E$. The following implications then hold:
\begin{enumerate}[label=\normalfont(\roman*)]
\item $x\leq y$ implies $x+z\le y+z $ and $\alpha x\leq \alpha y$ for all $\alpha\ge 0$ and any $x,y,z\in E$; and
\item $\Vert x\Vert \leq \Vert y\Vert$ whenever $\vert x\vert \leq \vert y\vert$.
\end{enumerate}
An element $x\in E$ is positive if $x\ge 0$. A positive cone in $E$ is denoted by $E_+$ (a notation also used for general ordered Banach spaces). The topological dual $E'$ of a Banach lattice $E$, under the usual dual norm and order, is also a Banach lattice. An operator $P\in \calB(E,F)$ is called positive if it maps $E_+$ into $F_+$, where $F$ is a Banach lattice generated by a cone $F_+$. The set of all such positive operators is denoted by $ \calB_+(E,F) $. This notation is also used when $E$ and $F$ are ordered Banach spaces. 

We will use the symbol $I$ generically to denote the identity operator on the corresponding space which will be clear from the context. Let $X$ be a Banach lattice and $A:\bm{D}({{A}})\subset X\to X$ be a closed densely defined linear operator generating a $C_0$-semigroup ${{T}}\coloneqq (T(t))_{t\in\R_+}$ on $X$. The spectral bound of ${{A}}$ is defined by $s({{A}})\coloneqq\sup\{\operatorname{Re}\lambda:\lambda\in\sigma({{A}})\}$, where $\sigma({{A}})$ denotes the spectrum of ${{A}}$, and the growth bound (type) of ${{T}}$ is defined by $\omega_0({{T}})\coloneqq\inf\{\omega\in\R:\text{$\exists M\ge 1$ such that $\|T(t)\|\le Me^{\omega t}$ for all $t\in\R_+$}\}$. It holds that $s({{A}})\le\omega_0({{T}})$ and $r(T(t))=e^{\omega_0({{T}})t}$, $t\in\R_+$, where $r(T(t))$ denotes the spectral radius of $T(t)$ (see, e.g., \cite[Proposition IV.2.2]{MR1721989}). A semigroup ${{T}}$ is said to be exponentially stable if its growth bound $\omega_0({{T}})<0$. If $s({{A}})=\omega_0({{T}})$, then we say that $T$ (or $A$) satisfies the spectrum-determined growth condition, and in this case the exponential stability of ${{T}}$ is obviously equivalent to the property $s({{A}})<0$. We let $X_{-1}^{{{A}}}$ denote the extrapolation space associated with $X$ and $A$, which is the completion of $X$ in the norm $\|x\|_{-1}^{{{A}}}\coloneqq \|R(\la,{{A}})x\|$, the resolvent operator $R(\la,{{A}})\coloneqq (\la I-{{A}})^{-1}$ for some arbitrary, but fixed, $\la\in\C$ in the resolvent set $\varrho({{A}})$ of ${{A}}$. Then we see that $X \subset X_{-1}^{{{A}}}$ and $X$ is dense and continuously embedded in $X_{-1}^{{{A}}}$. Furthermore, a $C_0$-semigroup ${{T}}$ extends uniquely to a $C_0$-semigroup ${{T}}_{-1}\coloneqq ( T_{-1}(t))_{t\in\R_+}$ on $X_{-1}^{{{A}}}$, whose infinitesimal generator ${{A}}_{-1}$ is an extension of ${{A}}$ to $X$, i.e.\ with domain $\bm{D}({{A}}_{-1})=X$; the semigroups ${{T}}$ and ${{T}}_{-1}$ are isomorphic to each other (see \cite[Section II.5]{MR1721989}). A semigroup ${{T}}$ on $X$ is called positive if $T(t)x\ge 0$ for all $t\in\R_+$ and for any $x\in X_+$. 

In our setting, $X$ is a Banach lattice with a natural positive cone $X_+$. However, it is not immediately clear how to extend this concept of positivity to the extrapolation space $X_{-1}^A$. Recall that if $A$ generates a positive $C_0$-semigroup on $X$, then we may define a positive cone in $X_{-1}^A$, denoted $(X_{-1}^A)_+$, as the closure of $X_+$ in the norm $\|\cdot\|_{-1}^{A}$ (see \cite[Definition 2.1]{MR3782646}). This definition ensures that $X_+\subset (X_{-1}^A)_+$. Furthermore, if $X$ is a real Banach lattice, then by \cite[Proposition 2.3]{MR3782646}
\begin{equation*}
X_+=X\cap (X_{-1}^{A})_+;
\end{equation*}
see more details in \cite[Section 2.2]{MR4910475}.

\subsection{Preliminary results on initial/boundary-value problems}\label{Sec:1.2}

Let ${X}$, ${U}$ be Banach spaces, called the \textit{state space} and \textit{boundary space}, respectively, and consider the following abstract initial/boundary-value problem on ${X}$:
\begin{equation}\label{dada}
\dot{z}(t)= {{A}} z(t),\quad{G}z(t)={K}z(t),\quad t>0,\quad z(0)=x.
\end{equation}
We will assume that ${{A}}:{Z}\sub {X}\to {X}$ is closed, ${Z}\sub {X}$ dense and continuously embedded in ${X}$, and $ {G},{K} :{Z}\to U$ are bounded linear boundary operators. 

In order to discuss the well-posedness of the initial/boundary-value problem \eqref{dada}, we first define the so-called \textit{system operator} $\mathcal{A}$ in $X$, which will also be used later in Section~\ref{Sec:2} in the statements of the main results; namely,
\begin{equation}\label{Big-A}
\calA x\coloneqq {A}x,\quad \bm{D}(\calA)= \left\{x\in {Z}:\: Gx =Kx\right\}.
\end{equation}
By standard semigroup theory, the system \eqref{dada} is well posed if and only if the system operator $\mathcal{A}$ generates a $C_0$-semigroup $\mathcal{T} \coloneqq (\mathcal{T}(t))_{t\in\R_+}$ on $X$. For this reason, we focus our attention in the sequel on the $C_0$-semigroup generation properties of $\mathcal{A}$. We impose the following two standard assumptions (which we assume throughout):
\begin{mainassumptions}\label{Assp11}
Consider the pair $({{A}},{G})$. The following assumptions hold:
\begin{enumerate}[label=\normalfont(\roman*)]
\item $\tilde{{{A}}}\coloneqq {{A}}\vert_{\ker {G}}$ generates a $C_0$-semigroup ${{T}}\coloneqq ({T}(t))_{t\in\R_+}$ on ${X}$; and
\item ${G}$ is onto.
\end{enumerate}
\end{mainassumptions}
Under Assumption \ref{Assp11} we have that the space ${Z}$ can be decomposed and related to the domain of $\tilde{{A}}$ as follows:
\begin{equation*}
{Z}=\ker {G} \oplus \ker ( \la I- {A} ),\quad \la\in \varrho(\tilde{{A}}).
\end{equation*}
To proceed, we will use the notion of a Dirichlet operator an define
\begin{equation*}\label{Dirichlet}
{D}_\la\coloneqq ({G}|_{\ker(\la I-{A})})^{-1}: {U}\to \ker(\la I-{A})\subseteq {X},\quad \la\in\varrho(\tilde{{A}}).
\end{equation*}
Hence, ${D}_\la$ exists and ${D}_\la\in\calB(U,X)$ (see \cite[Lemmas 1.2 and 1.3]{MR904952}), and we can relate the pair $(A,G)$ to a well-defined Dirichlet form by introducing the so-called \textit{boundary control operator}
\begin{equation*}\label{boundary-control}
{B}\coloneqq (\la I-\tilde{{A}}_{-1}){D}_\la\in\calB({U},{X}_{-1}^{\tilde{{A}}}),\quad \la\in \varrho(\tilde{{A}}).
\end{equation*}
Note that ${B}$ is independent of the choice of $\la$, $\operatorname{ran}{B}\cap {X}=\{0\}$, and that 
\begin{equation}\label{representation}
({A}-\tilde{{A}}_{-1})|_{Z}={B} {G}.
\end{equation}
Formally, we have the following definition. Throughout $X$ and $U$ are Banach lattices, and we put
\begin{equation*}
{C}\coloneqq {K}\vert_{\ker{G}},
\end{equation*}
the so-called \textit{observation operator}.
\begin{definition}\label{Well-posed-triple}
Let Assumption \ref{Assp11} be satisfied. Suppose that the operators $K$, $B$, and $T$ are positive. Then, for any $p\in [1,\infty)$, we say that the triple $(\tilde{{A}},{B},{C})$ is a positive $\bm{L}_p$-well-posed triple on $({U},{X},{U})$ if the following statements hold:
\begin{enumerate}[\normalfont(1)]
\item The pair $(\tilde{{A}}, {B})$ is positive $\bm{L}_p$-admissible, i.e., for some (and, hence, for all) $t>0$ and $u\in \bm{L}_{p,+}(0,t;{U})$, the input map corresponding to $(\tilde{{A}}, {B})$ is defined by
\begin{equation*}
\Phi_{t}u\coloneqq \int_{0}^{t} T_{-1}(t-r){B}u(r)dr\in {X}_+.
\end{equation*}
\item The pair $(C,\tilde{A})$ is positive $\bm{L}_p$-admissible, i.e., for some (and, hence, for all) $t >0$, there exists $\gamma\coloneqq \gamma(t)>0$ such that for any $x\in \bm{D}_+(\tilde{{A}})$,
\begin{equation*}
\Vert {C} T(\cdot)x\Vert_{\bm{L}_{p}(0,t;{U})}\leq \gamma^p\Vert x\Vert_X^p. 
\end{equation*}
\item For every $t>0$ there exists $\eta\coloneqq \eta(t)>0 $ such that for any $ u\in \mathring{\bm{W}}^{1}_{p,+}(0,t;{U})\coloneqq \{ u\in \bm{W}^{1}_{p,+}(0,t;{U}):\, u(0)=0\}$,
\begin{equation*}
\Vert{K} \Phi_\cdot u\Vert_{\bm{L}_{p}(0,t;{U})}\leq \eta \Vert u\Vert_{\bm{L}_{p}(0,t;{U})}.
\end{equation*}
\end{enumerate}
\end{definition}
The implications of the definition call for some comments. Consider the operator
\begin{equation*}\label{input.output.operator}
(\F u)(r)\coloneqq {K}\Phi_r u
\end{equation*}
for any $u\in \mathring{\bm{W}}^{1}_{p,+}(0,t;{U})$, $r\in [0,t]$, and $t\in\R_+$. If the triple $(\tilde{{A}},{B},{C})$ is a positive $\bm{L}_p$-well-posed triple, then $\F$ extends uniquely to a positive bounded linear operator on $\bm{L}_{p}(0,t;{U})$, still denoted by $\F$. Moreover, there exists a unique positive analytic $\calB({U})$-valued function $\mathbf{H}$, defined on $\C_\alpha\coloneqq\left\{\lambda\in \C: \operatorname{Re}\lambda > \omega_0(T)\right\}$, which is given by
\begin{equation}\label{transfer.fct}
\mathbf{H}(\lambda) = K R(\lambda,\tilde{A}_{-1}) B
\end{equation}
and is called the \textit{transfer function} of the triple $(\tilde{A},B,C)$. Since $K$ is an extension of the observation operator $C$, which is generally not unique, the transfer function $\mathbf{H}$ is defined only up to such an extension. Therefore, as all resulting functions are analytic, we do not distinguish between different representations of $\mathbf{H}$ (see, e.g., \cite[Section 4.6]{MR2154892} or \cite[Section 3]{StaffansWeiss2002}).

To better capture the role of the transfer function $\mathbf{H}$ in our well-posedness analysis, we introduce the following subclasses of $\bm{L}_p$-well-posed linear systems.
\begin{definition}\label{regularity}
Under the assumptions and conditions of Definition \ref{Well-posed-triple}, the triple $(\tilde{{A}},{B},{C})$ is called a positive $\bm{L}_p$-well-posed triple with zero feedthrough which is of weakly, strongly, or uniformly regular type if, respectively, we have for any $v\in {U}$, $u^*\in {U}'$, $\lim_{\la\to +\infty }\langle\mathbf{H}(\la)v,u^*\rangle_{U,U'}=0$, for any $u\in {U}$, $\lim_{\la\to +\infty }\Vert \mathbf{H}(\la)u\Vert=0$, or $\lim_{\la\to +\infty }\Vert\mathbf{H}(\la)\Vert_{\calB({U})}=0$, all referring to $\la\in\R$.
\end{definition}

In the strongly regular case, the extension of $\F$ admits the representation
\begin{equation*}
(\F u)(r)=C_\Lambda \Phi_r u
\end{equation*}
for any $u\in \bm{L}_{p}(0,t;{U})$, a.e.\ $r\in [0,t]$, and $t\in \R_+$. This extension is referred to as the (extended) \textit{input/output map} of the triple $(\tilde{A},B,C)$, e.g.\ see \cite[Section 3.3]{GANTOUH2026106307}. Here $C_\Lambda $ denotes the Yoshida extension of $C$ with respect to $\tilde{A}$, defined by
\begin{equation*}
C_\Lambda x := \lim_{\la \to +\infty}C\la R(\la,\tilde{A})x,\quad x\in \bm{D}(C_\Lambda):= \{x\in X: \text{$\lim_{\la \to +\infty}C\la R(\la,\tilde{A})x$ exists in $U$}\}.
\end{equation*}

We will also need the following concept of an \textit{admissible positive feedback operator}, related to Definition \ref{Well-posed-triple}.
\begin{definition}\label{Definition-admissible}
The identity operator $I$ on $U$ is called an admissible positive feedback operator for a positive $\bm{L}_p$-well-posed triple $(\tilde{{A}},{B},{C})$ with corresponding input/output map $\F$ if the inverse of $I-\F$ exists and is positive as an operator on $\calB(\bm{L}_p(0,\t;{U}))$ for some $\t>0$.
\end{definition}

We conclude this section with a generation result under domain perturbations of positive semigroup generators, which simultaneously establishes both the well-posedness and positivity of the initial/boundary-value problem \eqref{dada} (see \cite[Theorem A.1]{GANTOUH2026106307}). We state the result as a lemma.
\begin{lemma}\label{theorem-ABC}
Suppose that the triple $(\tilde{{A}},{B},{C})$ is a positive $\bm{L}_p$-well-posed weakly regular triple with zero feedthrough in the sense of Definition \ref{regularity} and $I$ is an admissible positive feedback operator in the sense of Definition \ref{Definition-admissible}. Then the system operator $\calA$ generates a positive $C_0$-semigroup $\calT\coloneqq (\calT(t))_{t\in\R_+}$ on $ {X} $ which can be represented in the form 
\begin{equation*}\label{variation}
\calT(t)x=T(t)x+\int_{0}^{t}{T}_{-1}(t-r){B}{C}_{\Lambda}\calT(r)x dr,\quad t\in\R_+,
\end{equation*}
for any $x\in {X}$. The initial/boundary-value problem \eqref{dada} admits a unique positive mild solution $z\in \bm{C}(\R_+;{X})$ which can can be defined by $\R_+\ni t\mapsto z(t)\coloneqq\calT(t)x$ for any $x\in X$. Furthermore, we have for the spectral bound of $\calA$
\begin{equation}\label{spectral-bound}
s(\calA)\ge s(\tilde{{A}}).
\end{equation}
\end{lemma}

\section{Main results}\label{Sec:2} 

In this section we present our main results on the well-posedness and exponential stability of the $\mathbf{aBCS}$. The analysis proceeds along two complementary approaches. While the first approach is based on the feedback-theoretic framework presented in Section \ref{Sec:1.2}, the second approach relies on perturbation techniques for the so-called \textit{resolvent positive operators} as developed in \cite{MR4711370} (appropriate definitions will be given later). These approaches will be loosely termed in the sequel the ``feedback theory approach'' and the ``perturbation approach'', respectively.

Throughout we assume the spaces $X_k$ and $U_k$, $k=1,2$, to be Banach lattices. To write the $\mathbf{aBCS}$ as an abstract initial/boundary-value problem of the form \eqref{dada}, the state and boundary spaces are decomposed
\begin{equation*}
X = X_1\times X_{2}, \quad U= U_1 \times U_2
\end{equation*}
and are Banach lattices with the norms
\begin{equation*}
\left\|\left(\begin{matrix}
x_1\\
x_2
\end{matrix}\right)\right\|_{X}\coloneqq\|x_1\|_{X_1}+\|x_2\|_{X_1}, \quad \left\|\left(\begin{matrix}
u_1\\
u_2
\end{matrix}\right)\right\|_{U}\coloneqq\|u_1\|_{U_1}+\|u_2\|_{U_1}.
\end{equation*}
In terms of the state function
\begin{equation*}
t\mapsto z(t)=\left(\begin{matrix}
z_1(t)\\
z_2(t)
\end{matrix}\right),
\end{equation*}
the $\mathbf{aBCS}$ can then be written in the form \eqref{dada} with initial state $z(0)\coloneqq(x_1,x_2)^\top$, where the operators ${{A}}$, $G$, and $K$ are defined in $X$ with the same domain
\begin{equation*}
\bm{D}({{A}}_1)\times \bm{D}({{A}}_2)
\end{equation*}
by the block operator matrices
\begin{gather*}
{{A}}\coloneqq\left(\begin{matrix}
{{A}}_1 & 0 \\
0 & {{A}}_2
\end{matrix}\right),\quad G\coloneqq\left(\begin{matrix}
G_1 & 0 \\
0 & G_2
\end{matrix}\right),\quad K\coloneqq\left(\begin{matrix}
0 & K_1 \\
K_2 & 0
\end{matrix}\right).
\end{gather*}
Therefore, the system operator $\mathcal{A}$ in $X$ is as defined in \eqref{Big-A}, with $Z=\bm{D}(A)\coloneqq\bm{D}({{A}}_1)\times \bm{D}({{A}}_2)$ used in the definition of $\bm{D}(\mathcal{A})$.

\subsection{Feedback theory approach}

Before we proceed to state the main results of this subsection we recall that Assumption \ref{Assp11} mentioned in the context of the pair $(A,G)$ continues to hold for $(A_k,G_k)$, $k=1,2$. Thus, in line with Assumption \ref{Assp11}, we make the following assumptions.
\begin{mainassumptions}\label{Assp12}
Consider the pair $(A_k,G_k)$, $k=1,2$. For $k=1,2$, the following assumptions hold:
\begin{enumerate}[label=\normalfont(\roman*)]
\item $\tilde{{A}}_k\coloneqq {A}_k\vert_{\ker {G}_k}$ generates a positive $C_0$-semigroup ${{T}}_k\coloneqq ({T}_k(t))_{t\in\R_+}$ on ${X}_k$; and
\item ${G}_k$ is onto.
\end{enumerate}
\end{mainassumptions}

With Assumption \ref{Assp12} it can be directly seen from the paragraph following the statement of Assumption \ref{Assp11} that for $k=1,2$ there are Dirichlet and boundary control operators 
\begin{equation*}
D_\la^k\in \calB(U_k,X_k),\quad B_k\in\calB(U_k,(X_k)_{-1}^{\tilde{A}_k})
\end{equation*}
corresponding to the pairs $(A_k,G_k)$ which can be defined, respectively, by
\begin{equation*}
D_\la^k\coloneqq (G_k|_{\ker(\la I-\tilde{A}_k)})^{-1},\quad B_k\coloneqq (\la I-(\tilde{A}_k)_{-1})D_\la^k,\quad \la\in \varrho(\tilde{A}_k).
\end{equation*}
Thus there are again corresponding well-defined Dirichlet forms. Let us, as before, put
\begin{equation*}
C_k\coloneqq K_k\vert_{\ker{G}_k},\quad k=1,2.
\end{equation*}
Bringing Assumption \ref{Assp12} and the definitions above together, we now are in a position to state our first main result, which essentially guarantees the existence of unique positive solutions of the $\mathbf{aBCS}$. The proof which we give is a direct application of Lemma \ref{theorem-ABC}. 
\begin{theorem}\label{theorem-dynamic-boundary}
Let Assumption \ref{Assp12} be satisfied and assume that $K_k$ and $B_k$, $k=1,2$, are positive operators. Suppose that the triples $(\tilde{A}_1,B_1, C_1)$, $(\tilde{A}_2,B_2,C_2)$ with corresponding input/output maps $\F_{1}$ and $\F_2$, respectively, are positive $\bm{L}_p$-well-posed weakly regular triples with zero feedthrough in the sense of Definition \ref{regularity}. If the spectral radius $r(\F_1\F_2)<1$ or, equivalently, $r(\F_2\F_1)<1$, then the system operator $\calA$ generates a positive $C_0$-semigroup $\calT\coloneqq (\calT(t))_{t\in\R_+}$ on $X$ and \eqref{dada} has a unique positive mild solution $z\in \bm{C}(\R_+;{X})$ for any $x\in X$ which can be represented in the form
\begin{equation}\label{eqwr4532}
{z}(t) =\calT(t)x,\quad t\in\R_+,
\end{equation}
and which gives the unique solutions of the $\mathbf{aBCS}$.
\end{theorem}
\begin{proof}
Under Assumption \ref{Assp12}, we can always define
\begin{equation}\label{restriction-operator}
\tilde{A}\coloneqq A,\quad \bm{D}(\tilde{A})\left(\coloneqq \ker G\right)= \bm{D}(\tilde{A}_1)\times \bm{D}(\tilde{A}_2)
\end{equation}
such that $\tilde{A}$ generates a positive $C_0$-semigroup $T\coloneqq \left(T(t)\right)_{t\geq0}$ on $X$. Hence,
\begin{equation*}
T=\left(\begin{matrix}
T_1 & 0 \\
0 & T_2
\end{matrix}\right)
\end{equation*}
and moreover, since there is a Dirichlet operator on $U$ given by
\begin{equation*}
D_\la\coloneqq\left(\begin{matrix}
D_\la^{1} & 0 \\
0 & D_\la^{2}
\end{matrix}\right)
\end{equation*}
for each $\la \in \varrho(\tilde{A})\coloneqq\varrho(\tilde{A}_1)\cap \varrho(\tilde{A}_2)$, we have for each such $\la$ that there is a boundary control operator on $U$ which is defined by
\begin{equation*}
B\coloneqq (\la I-\tilde{A}_{-1})D_\la\in \calB(U, X_{-1}^{\tilde{A}}),
\end{equation*}
where, we recall, $\tilde{A}_{-1}$ is an extension of $\tilde{A}$ to $X$. Therefore
\begin{equation*}
B=\left(\begin{matrix}
B_1 & 0 \\
0 & B_2
\end{matrix}\right)
\end{equation*}
and we obtain for the input map corresponding to the pair $(\tilde{{A}}, {B})$
\begin{equation}\label{input-maps}
\Phi_{t}\left(\begin{matrix}
u_1\\
u_2
\end{matrix}\right)=\left(\begin{matrix}
\Phi_{t}^1u_1\\
\Phi_{t}^2u_2
\end{matrix}\right)
\end{equation}
for all $t\in\R_+$ and any $(u_1,u_2)^\top\in \bm{L}_p(\R_+, U)$. This establishes the positive $\bm{L}_p$-admissibility of the pair $(\tilde{A},B)$ in the sense of Definition \ref{Well-posed-triple}, wherein we have for $k=1,2$ the input map $\Phi_{t}^k$ corresponding to the pair $(\tilde{A}_k,B_k)$ given by
\begin{equation*}
\Phi^k_{t}u_k= \int_{0}^{t} (T_k)_{-1}(t-r){B}_ku_k(r)dr
\end{equation*}
for any $u_k\in \bm{L}_{p,+}(0,t;U_k)$. Define now the observation operator 
\begin{equation*}
C\coloneqq K\vert_{ \bm{D}(\tilde{A})}=\left(\begin{matrix}
0 & C_2\\
C_1& 0
\end{matrix}\right).
\end{equation*}
The pairs $(C_1,\tilde{A}_1)$, $(C_2,\tilde{A}_2)$ are positive $\bm{L}_p$-admissible, implying that the pair $(C,\tilde{A})$ is also positive $\bm{L}_p$-admissible. For its Yoshida extension, we have
\begin{equation}\label{Yoshida-extension}
C_\Lambda\coloneqq\left(\begin{matrix}
0 & (C_2)_\Lambda\\
(C_1)_\Lambda& 0
\end{matrix}\right), \quad \bm{D}(C_\Lambda)=\bm{D}((C_1)_\Lambda)\times \bm{D}
((C_2)_\Lambda),
\end{equation}
where $(C_1)_\Lambda$, $(C_2)_\Lambda$ are the Yoshida extensions of $C_1$ and $C_2$ with respect to $\tilde{A}_1$ and $\tilde{A}_2$, respectively. Finally we define the operator $\F$ by
\begin{equation*}
\left(\F\left(\begin{matrix}
u_1\\
u_2
\end{matrix}\right)\right)(t)\coloneqq K \Phi_{t}\left(\begin{matrix}
u_1\\
u_2
\end{matrix}\right)
\end{equation*}
for any $(u_1,u_2)^\top\in \mathring{\bm{W}}^{1}_{p,+}(0,\t; U)$, $t\in[0,\t]$, and $\t\in\R_+$, which in view of \eqref{input-maps} implies
\begin{equation*}
\left(\F\left(\begin{matrix}
u_1\\
u_2
\end{matrix}\right)\right)(t)= K \left(\begin{matrix}
\Phi_{t}^1u_1\\
\Phi_{t}^2u_2
\end{matrix}\right)=\left(\begin{matrix}
K_2\Phi_{t}^2u_2\\
K_1\Phi_{t}^1u_1
\end{matrix}\right).
\end{equation*}
We have that the triples $(\tilde{A}_1,B_1, C_1)$, $(\tilde{A}_2,B_2,C_2)$ are hence positive $\bm{L}_p$-well-posed triples and so it follows from Definition \ref{Well-posed-triple} that there exists $\eta\coloneqq \eta(\t)$ for every $\tau>0$ such that for any $u\in \mathring{\bm{W}}^{1}_{p,+}(0,\t;U)$,
\begin{equation*}
\Vert \F u\Vert_{\bm{L}_p(0,\t;U)}\leq \eta \Vert u\Vert_{\bm{L}_p(0,\t;U)}.
\end{equation*}
Thus all the conditions of Definition \ref{Well-posed-triple} are satisfied, and the triple $(\tilde{A},B,C)$ is a positive $\bm{L}_p$-well-posed triple on $(U,X,U)$ with input/output map $\F$ given by
\begin{equation*}
\F=\left(\begin{matrix}
0&\F_2\\
\F_1&0
\end{matrix}\right),
\end{equation*}
$\F_{1}$ and $\F_2$ being the input/output maps corresponding to the triples $(\tilde{A}_1,B_1,C_1)$, $(\tilde{A}_2,B_2,C_2)$, respectively. Consequently for the transfer function of the triple $(\tilde{A},B,C)$, 
\begin{equation*}
\mathbf{H}=\left(\begin{matrix}
0&\mathbf{H}_2\\
\mathbf{H}_1&0
\end{matrix}\right),
\end{equation*}
where $\mathbf{H}_1$ and $\mathbf{H}_2$ are the transfer functions of the triples $(\tilde{A}_1,B_1,C_1)$, $(\tilde{A}_2,B_2,C_2)$, respectively, and hence, $(\tilde{A}_1,B_1,C_1)$ and $(\tilde{A}_2,B_2,C_2)$ being weakly regular triples, for $\la\in\R$ and any $v\in U$, $u^*\in U'$,
\begin{equation*}
\lim_{\la\to +\infty }\langle\mathbf{H}(\la)v,u^*\rangle_{U,U'}=0.
\end{equation*}
It follows that the triple $(\tilde{A},B,C)$ is a positive $\bm{L}_p$-well-posed weakly regular triple with zero feedthrough in the sense of Definition \ref{regularity}. Let us write
\begin{equation*}
I-\F=\left(\begin{matrix}
I-\F_1\F_2&-\F_2\\
0 & I
\end{matrix}\right)\left(\begin{matrix}
I &0\\
-\F_1 & I 
\end{matrix}\right).
\end{equation*}
The identity operator $I$ on $U$ appearing in the left side is an admissible positive feedback operator for the triple $(\tilde{{A}},{B},{C})$ with corresponding input/output map $\F$ (in the sense of Definition \ref{Definition-admissible}) if and only if the inverse of $I-\F_1\F_2$ exists and is positive as an operator on $\calB_+(\bm{L}_p(0,\t;U_1),\bm{L}_p(0,\t;U_2))$ for some $\tau>0$. This is fulfilled in view of the conditions imposed on the spectral radii in the theorem. Applying Lemma \ref{theorem-ABC} with $\calA$ as defined by \eqref{Big-A}, we conclude that the system operator $\calA$ generates a positive $C_0$-semigroup $\calT\coloneqq (\calT(t))_{t\in\R_+}$ on $X$ which can be represented in the form \eqref{variation}. Moreover, \eqref{dada} admits a unique positive mild solution $z\in \bm{C}(\R_+;{X})$ given by \eqref{eqwr4532}. This completes the proof of the theorem.
\end{proof}

\begin{remark}
Note from \eqref{restriction-operator} that the domain $\bm{D}(\tilde{A})$ of $\tilde{A}$ is diagonal. Then, with \eqref{input-maps} and \eqref{Yoshida-extension}, the solutions $z=z(t)$ can be written component-wise for any $x_1\in X_1$ and $x_2\in X_2$ as
\begin{equation*}
z_1(t)=T_1(t)x_1+ \Phi_t^1(C_2)_\Lambda z_2(\cdot),\quad z_2(t)=T_2(t)x_2+ \Phi_t^2(C_1)_\Lambda z_1(\cdot),\quad t\in\R_+.
\end{equation*}
\end{remark}

The next result consists of a characterization of the spectral properties of the system operator $\calA$ along with its resolvent operator computation. 
\begin{proposition}\label{Spectrum}
Let Assumption \ref{Assp12} be satisfied and let it be further assumed that $K_k$, $k=1,2$, is a positive operator. For $k=1,2$, suppose that $D_\lambda^k$ is a positive operator for all $\lambda > s(\tilde{A}_k)$. Suppose also that $\mathcal{A}$ generates a positive $C_0$-semigroup $\mathcal{T} := (\mathcal{T}(t))_{t\in\R_+}$ on $X$. For $\lambda \in \varrho(\tilde{A}_1) \cap \varrho(\tilde{A}_2)$, define
\begin{equation*}
\mathbb{A}_\la\coloneqq K_1 D_\la^1 K_2 D_\la^2.
\end{equation*}
Then the following conditions are equivalent:
\begin{enumerate}[\normalfont(i)]
\item\label{implSpectrum01} $\max\big\{s(\tilde{A}_1), s(\tilde{A}_2)\big\}<\lambda$ and $r(\mathbb{A}_\lambda)<1$; and
\item\label{implSpectrum02} $s(\mathcal{A})<\lambda$.
\end{enumerate}
Moreover, for all $\lambda>s(\mathcal{A})$, 
\begin{equation}\label{resolvent-operator}
R(\la, \calA)=
\left(\begin{matrix}
\left(I+ D_\la^1 K_2 \Delta_\la\right)R(\la,\tilde{A}_1) & D_\la^1 K_2\left(I+ \Delta_\la D_\la^1 K_2\right)R(\la,\tilde{A}_2)\\
\Delta_\la R(\la,\tilde{A}_1) & I+\Delta_\la D_\la^1K_2 R(\la,\tilde{A}_2)
\end{matrix}\right),
\end{equation}
wherein $\Delta_\la\coloneqq D_\la^2\left(I-\mathbb{A}_\la\right)^{-1}K_1$.
\end{proposition} 
\begin{proof}
Let $\la >\max\{s(\tilde{A}_1),s(\tilde{A}_2)\}$ and $x\in {\bm D}(\calA) $. From the representation formula \eqref{representation}, we obtain
\begin{align}
(\la I-\calA )x &=(\la I- A)x\notag\\
&=(\la I- \tilde{A}_{-1}-B G)x\notag\\
&=(\la I- \tilde{A})(x-D_\la Gx)\notag\\
&=(\la I- \tilde{A})(x-D_\la Kx), \label{resolvent.charact}
\end{align}
where the third equality holds because $x-D_\la Gx\in {\bm D}(\tilde{A})$. Thus,
\begin{equation}\label{spectral-chara}
\la \in \varrho(\calA) \iff 1\in \varrho( D_\la K) \iff 1\in \varrho(K D_\la ),
\end{equation}
where we have used that $K D_\lambda$ and $D_\lambda K$ share the same nonzero spectrum (since $(I - D_\lambda K)^{-1}=I+D_\lambda (I - K D_\lambda)^{-1} K$). Next, for $\la >\max\{s(\tilde{A}_1),s(\tilde{A}_2)\}$, we obtain the factorization
\begin{equation}\label{cara2}
I-K D_\la=\left( \begin{matrix}
I & 0\\
-K_1 D_\la^1 & I-\A_\la
\end{matrix}\right)\left(\begin{matrix}
I & -K_2 D_\la^2\\
0& I
\end{matrix}\right).
\end{equation}
Thus $I-KD_\lambda$ is invertible if and only if $I-\A_\lambda$ is invertible. Since $\A_\lambda$ is a positive operator, this is equivalent to $r(\A_\lambda)<1$. 
Combining this with \eqref{spectral-chara}, we obtain
\begin{equation}\label{condition.invert}
\lambda\in \varrho(\mathcal{A}) \iff r(\A_\lambda)<1.
\end{equation}
If \ref{implSpectrum01} holds, then $\lambda\in \varrho(\mathcal{A})$, hence $s(\mathcal{A})\le \lambda$. If $s(\mathcal{A})=-\infty$, the positivity of $\mathcal{T}$ implies that $R(\lambda,\mathcal{A})$ is positive, and therefore, by \cite[Theorem C-III.1.1-(b)]{MR839450}, $s(\mathcal{A})<\lambda$. If $s(\mathcal{A})>-\infty$, the positivity of $\mathcal{T}$ implies that $s(\mathcal{A})\in \sigma(\mathcal{A})$ (cf. \cite[Theorem C-III.1.1-(a)]{MR839450}). Hence, $s(\mathcal{A})<\lambda$. This proves \ref{implSpectrum01}$\implies$\ref{implSpectrum02}. For the proof of \ref{implSpectrum02}$\implies$\ref{implSpectrum01}, let $s(\calA)<\lambda$. Then, by \eqref{spectral-bound},
\begin{equation*}
s(\calA) > \max\{s(\tilde{A}_1),s(\tilde{A}_2)\},
\end{equation*}
and hence $\lambda>\max\{s(\tilde{A}_1),s(\tilde{A}_2)\}$. Moreover, from \eqref{condition.invert} we obtain $r(\A_\lambda)<1$, and we have the desired result.

Now, using \eqref{resolvent.charact}, it follows that for all $\lambda>s(\mathcal{A})$,
\begin{align}
R(\la, \calA)&= (I-D_\lambda K)^{-1} R(\la ,\tilde{A})\notag\\
&=(I+{D}_\la (I-KD_\la)^{-1}{K})R(\la ,\tilde{{A}}).\label{resolvent}
\end{align}
A direct computation using \eqref{cara2} gives
\begin{equation*}
(I-K D_\la )^{-1}=\left(\begin{matrix}
I+ K_2 D_\la^2\left(I-\mathbb{A}_\la\right)^{-1}K_1 D_\la^2& K_2 D_\la^2\left(I-\mathbb{A}_\la\right)^{-1}\\
\left(I-\mathbb{A}_\la\right)^{-1}K_1 D_\la^1& \left(I-\mathbb{A}_\la\right)^{-1}
\end{matrix}\right).
\end{equation*}
Substituting this into the resolvent expression \eqref{resolvent} yields \eqref{resolvent-operator}, completing the proof.
\end{proof}

\begin{remark}
It is important to note that the proof of Proposition \ref{Spectrum} does not rely on the $\bm{L}_p$-well-posedness and regularity of the triples $(\tilde{A}_1,B_1,C_1)$ and $(\tilde{A}_2,B_2,C_2)$. Indeed, on the one hand, we have only used Assumption \ref{Assp12} along with the positivity of the operators $D_\lambda^k$ and $K_k$, $k=1,2$, which implies by \cite[Theorem II.5.3]{MR423039} that $\mathbb{A}_\lambda \in \mathcal{B}(U_2,U_1)$ and hence that $r(\mathbb{A}_\lambda)$ is well defined. On the other hand, we have relied on the fact that the system operator $\mathcal{A}$ generates a positive $C_0$-semigroup on $X$. This makes the spectral characterization in Proposition \ref{Spectrum} independent of the method used to prove the $C_0$-semigroup generation property of $\mathcal{A}$. Therefore, the results established in Proposition \ref{Spectrum} are robust in the sense that they do not depend on the specific well-posedness framework.
\end{remark}

\begin{remark}\label{rem.chara}
It is worth noting that Proposition \ref{Spectrum} provides a complete characterization of the spectral bound of $\mathcal{A}$. More precisely, we have
\begin{equation}\label{sharp-characterization}
s(\mathcal{A})=\inf\left\{\omega<\lambda: r(\A_\lambda)<1\right\},
\end{equation}
wherein $\omega \coloneqq\max\big\{s(\tilde{A}_1), s(\tilde{A}_2)\big\}$. In fact, let $\omega<\mu<\lambda$. Then,
\begin{equation*}
D_\lambda^k\le D_\mu^k,\quad k=1,2,
\end{equation*}
so that $\mathbb{A}_\lambda\le \mathbb{A}_\mu$ (since $K_1$, $K_2$ are positive). Hence, the mapping $\lambda \mapsto \mathbb{A}_\lambda$ is decreasing. By positivity of operators on Banach lattices, $\lambda \mapsto r(\mathbb{A}_\lambda)$ is monotone with respect to the order $\le$ and, thus, is decreasing. Hence, the set $\{\omega<\lambda:r(\A_\lambda)<1\}$ is nonempty and bounded below by $\omega$, so its infimum exists. By Proposition \ref{Spectrum}, $r(\A_\lambda)<1$ for all $\lambda>s(\mathcal{A})$, which implies \eqref{sharp-characterization}.
\end{remark}

An immediate consequence of Proposition \ref{Spectrum} is that a necessary condition for the system operator $\mathcal{A}$ to generate a positive $C_0$-semigroup is that the spectral radius $r(\mathbb{A}_\lambda) < 1$ for some $\lambda > \max\{\omega_0(T_1), \omega_0(T_2)\}$. When the underlying spaces are Hilbert lattices, this condition is also sufficient. In particular, we have the following result.
\begin{proposition}\label{sharp-wellposedness-cond}
For $k=1,2$, let $X_k$ and $U_k$ be Hilbert lattices. Let Assumption \ref{Assp12} be satisfied and let it be further assumed that $K_k$ and $B_k$, $k=1,2$, are positive operators. Suppose the triples $(\tilde{A}_1,B_1, C_1)$, $(\tilde{A}_2,B_2,C_2)$ to be positive $\bm{L}_2$-well-posed weakly regular triples with zero feedthrough in the sense of Definition \ref{regularity}. Then the following statements are equivalent:
\begin{enumerate}[\normalfont(1)]
\item\label{implCor01} The system operator $\calA$ generates a positive $C_0$-semigroup $\calT\coloneqq (\calT(t))_{t\in\R_+}$ on $X$.
\item\label{implCor02} $r(\A_{\lambda_0})<1$ for some $\lambda_0>\max\big\{\omega_0(T_1), \omega_0(T_2)\big\}$.
\end{enumerate}
\end{proposition}
\begin{proof}
The implication \ref{implCor01}$\implies$\ref{implCor02} follows from Proposition \ref{Spectrum}. In order to prove the converse implication, it suffices (by Lemma \ref{theorem-ABC} and the proof of Theorem \ref{theorem-dynamic-boundary}) to show that the identity operator $I$ on $U$ is an admissible positive feedback operator for the triple $(\tilde{A},B,C)$ defined in the proof of Theorem \ref{theorem-dynamic-boundary}. In the Hilbert space setting, this is equivalent to the condition (cf.\ \cite[Remark 3.4]{GANTOUH2026106307}) that for some $\lambda_0>\omega_0(T)$,
\begin{equation*}
r(\mathbf{H}(\lambda_0))<1,
\end{equation*}
where $\mathbf{H}$ is the transfer function of the triple $(\tilde{A},B,C)$. Since $(\tilde{A},B,C)$ is $\bm{L}_2$-well-posed (see the proof of Theorem \ref{theorem-dynamic-boundary}), it follows from \eqref{transfer.fct} that
\begin{equation*}
\mathbf{H}(\lambda) = K D_\lambda,\quad \operatorname{Re}\lambda>\omega_0(T).
\end{equation*}
Using the factorization \eqref{cara2}, \ref{implCor02} implies that $r(K D_{\lambda_0})<1$ and hence that $r(\mathbf{H}(\lambda_0))<1$. Therefore, the identity operator $I$ on $U$ is an admissible positive feedback operator for the triple $(\tilde{A},B,C)$. Applying Lemma \ref{theorem-ABC}, we conclude that $\calA$ generates a positive $C_0$-semigroup on $X$. This prove \ref{implCor02}$\implies$\ref{implCor01}. 
\end{proof}

\begin{remark}\label{test-remark1}
We note that the above result yields a useful and verifiable test for $C_0$-semigroup generation in terms of the transfer function of the $\mathbf{aBCS}$, in which the boundary conditions are encoded into the the boundary coupling operators. In other words, provided that the underlying setting is that of Hilbert lattices, the above result provides a very simple test for identifying those boundary couplings that guarantee the well-posedness of the $\mathbf{aBCS}$.
\end{remark}

We can now state and prove the following stability result.
\begin{theorem}\label{stability-result}
Consider the setting of Theorem \ref{theorem-dynamic-boundary}. Then the following conditions are equivalent:
\begin{enumerate}[\normalfont(i)]
\item\label{stability-resulta} $s(\calA)<0 $; and
\item\label{stability-resultb} $s(\tilde{A}_1),s(\tilde{A}_2)<0$ and $r(\A_0)<1$.
\end{enumerate}
Moreover, if any of the above conditions hold, then there exist constants $J_{x_1,x_2},\alpha>0$ such that
\begin{equation}\label{solution-estimate}
\| z_1(t)\|_{X_1}+\| z_2(t)\|_{X_2}\le J_{x_1,x_2} e^{-\alpha t}\left\| \left(\begin{matrix}
x_1\\
x_2
\end{matrix}\right)\right\|_{X}
\end{equation}
for all $t\in\R_+$ and any $x_1\in {\bm D}(A_1)$, $x_2\in {\bm D}(A_2)$ satisfying the coupling conditions $G_1 x_1=K_2 x_2$ and $ G_2 x_2=K_1 x_1$.
\end{theorem}
\begin{proof}
The equivalence between \ref{stability-resulta} and \ref{stability-resultb} follows directly from Proposition \ref{Spectrum}. If we suppose that either either \ref{stability-resulta} or \ref{stability-resultb} holds, then, according to \cite[Corollary C-IV.1.4]{MR839450}, there exist constants $J_x,\alpha>0$ such that for all $t\in\R_+$ and any $x\in {\bm D}(\calA)$,
\begin{equation*}
\| \calT(t)x\|_{X}\le J_x e^{-\alpha t}\| x\|_{ X}.
\end{equation*}
Combining this with \eqref{eqwr4532} yields the estimate \eqref{solution-estimate}, thereby proving the theorem.
\end{proof}

The following corollary answers the question as to what happens to the result of Theorem \ref{stability-result} if the spectrum-determined growth condition holds.
\begin{corollary}\label{Cor2}
Consider the setting of Theorem \ref{theorem-dynamic-boundary}. Assume that $\calT$ satisfies the spectrum-determined growth condition, i.e., $s(\calA)=\omega_0(\calT)$. If 
\begin{equation*}\label{spectral-condition}
s(\tilde{A}_1),s(\tilde{A}_2)<0,\quad r(\A_0)<1,
\end{equation*}
then there exist constants $J, \alpha>0$ such that for all $t\in\R_+$ and any $x_1\in X_1$, $x_2\in X_2$,
\begin{equation*}
\| z_1(t)\|_{X_1}+\| z_2(t)\|_{X_2}\le J e^{-\alpha t}\left(\| x_1\|_{X_1}+\| x_2\|_{X_2}\right).
\end{equation*}
\end{corollary}

\subsection{Perturbation approach}

In this subsection we present an alternative characterization of exponential stability, which is based on the perturbation approach for resolvent positive operators in \cite{MR4711370}. For this purpose, we recall that a linear operator $Q$ is called \textit{resolvent positive} if there exists $\omega \in \R$ such that $(\omega, \infty) \subset \rho(Q)$ and $R(\lambda,Q) \ge 0$ for all $\lambda > \omega$. Every infinitesimal generator of a positive $C_0$-semigroup is resolvent positive; however, the converse is not necessarily true (see, e.g., \cite[Section 3]{MR872810}).

We begin with a generation result for the system operator $\mathcal{A}$ in $X$.
\begin{theorem}\label{Result1}
For $k=1,2$, let the following assumptions hold:
\begin{enumerate}[label=\normalfont(\roman*)]
\item\label{cond1} $\ker G_k$ is dense in $X_k$ and ${G}_k$ is onto;
\item\label{cond2} $\tilde{A}_k\coloneqq {A}_k\vert_{\ker {G}_k}$ is resolvent positive and, for some $\lambda_0^k > s(\tilde{A}_k)$, there exists $c_k > 0$ such that
\begin{equation*}
\Vert R(\lambda, \tilde{A}_k)x_k \Vert_{X_k} \geq c_k \Vert x_k \Vert_{X_k}
\end{equation*}
for all $\lambda \ge \lambda_0^k$ and any $x_k \in (X_k)_+$; and
\item\label{cond3} $K_k$ is positive, and $D_\lambda^k$ is positive for large enough $\lambda\in\R$.
\end{enumerate}
Then $\calA$ generates a positive $C_0$-semigroup $\calT\coloneqq (\calT(t))_{t\in\R_+}$ on $X$ if and only if
\begin{equation}\label{spectral.cond}
r(K_1 D_{\lambda_1}^1 K_2 D_{\lambda_1}^2)<1
\end{equation}
or, equivalently, $r(K_2 D_{\lambda_1}^2 K_1 D_{\lambda_1}^1)<1$ for some $\lambda_1>\max\{s(\tilde{A}_1),s(\tilde{A}_2)\}$. In this case the spectrum-determined growth condition $\omega_0(\calT)=s(\mathcal{A})$ holds.
\end{theorem}
\begin{proof}
The proof of necessity is obvious from Proposition \ref{Spectrum}. For the sufficiency, define the operator
\begin{equation*}
\tilde{A}\coloneqq A,\quad \bm{D}(\tilde{A})\coloneqq \ker G.
\end{equation*}
Then,
\begin{equation*}
\tilde{A}=\left(\begin{matrix}
\tilde{A}_{1} & 0 \\
0 & \tilde{A}_{2}
\end{matrix}\right), \quad \bm{D}(\tilde{A})=\bm{D}(\tilde{A}_1)\times \bm{D}(\tilde{A}_2).
\end{equation*}
Under our assumptions \ref{cond1} and \ref{cond2}, we see that $\tilde{A}_1$ and $\tilde{A}_2$ are densely defined resolvent positive operators. Hence, $\tilde{A}$ is a densely defined resolvent positive operator as well, and we have, for $\la \in \varrho(\tilde{A})\coloneqq\varrho(\tilde{A}_1)\cap \varrho(\tilde{A}_2)$,
\begin{equation*}
R(\lambda,\tilde{A})=\left(\begin{matrix}R(\lambda,\tilde{A}_1) & 0\\ 0 & R(\lambda,\tilde{A}_2)\end{matrix}\right).
\end{equation*}
Moreover, for all $\lambda \ge \max\{\lambda_0^1,\lambda_0^2\}$ and any $x\in X_+$,
\begin{equation*}
\Vert R(\lambda, \tilde{A})x \Vert_X \geq c \Vert x \Vert_X,
\end{equation*}
some constant $c\coloneqq \inf\{c_1,c_2\}$. Thus, by \cite[Theorem 2.5]{MR872810}, $\tilde{A}$ generates a positive $C_0$-semigroup on $X$. By assumption \ref{cond1}, there is a Dirichlet operator $D_\la\in \mathcal{B}(U,X)$ given by
\begin{equation*}
D_\la\coloneqq\left(\begin{matrix}
D_\la^{1} & 0 \\
0 & D_\la^{2}
\end{matrix}\right)
\end{equation*}
for each $\la \in \varrho(\tilde{A})$. By assumption \ref{cond3}, $D_\la$ is positive for $\lambda\in \R$ large enough, and hence, by \cite[Proposition 4.3]{MR4910475}, for all $\lambda>s(\tilde{A})$. Furthermore, using the factorization \eqref{cara2}, we obtain for all $\la >\max\{s(\tilde{A}_1),s(\tilde{A}_2)\}$
\begin{equation*}
I-K D_\la=\left( \begin{matrix}
I & 0\\
-K_1 D_\la^1 & I-\A_\la
\end{matrix}\right)\left(\begin{matrix}
I & -K_2 D_\la^2\\
0& I
\end{matrix}\right),
\end{equation*} 
wherein $\mathbb{A}_\lambda \coloneqq K_1 D_\lambda^1 K_2 D_\lambda^2$. Since $K_1$ and $K_2$ are positive, it follows that $r(K D_\la)<1$ if and only if $r(\mathbb{A}_\lambda)<1$, the latter inequality holding by virtue of \eqref{spectral.cond}. Therefore, by \cite[Theorem 1]{MR4711370}, we conclude that the system operator $\mathcal{A}$ generates a positive $C_0$-semigroup $\calT\coloneqq (\calT(t))_{t\in\R_+}$ on $X$ satisfying $\omega_0(\calT)=s(\mathcal{A})$. The proof is complete.
\end{proof}

\begin{remark}\label{test-remark2}
It is worth noting that the above theorem extends Proposition \ref{sharp-wellposedness-cond} to the Banach lattice setting. In particular, Theorem \ref{Result1} provides a simple and readily verifiable test on the boundary coupling operators for the $C_0$-semigroup generation property of the system operator $\calA$ in this general setting.
\end{remark}

We are now in a position to state and prove the following stability result.
\begin{theorem}\label{Result2}
Let the assumptions of Theorem \ref{Result1} be satisfied and suppose that
\begin{equation}\label{Exp.cond}
\max\{s(\tilde{A}_1),s(\tilde{A}_2)\}<0,\quad r(\A_0)<1.
\end{equation}
Then the $C_0$-semigroup $\calT$ is exponentially stable.
\end{theorem}
\begin{proof}
As a result of the properties in \eqref{Exp.cond}, it follows from Theorem \ref{Result1} that the system operator $\mathcal{A}$ generates a positive $C_0$-semigroup $\calT\coloneqq (\calT(t))_{t\in\R_+}$ on $X$ satisfying the spectrum-determined growth condition $\omega_0(\mathcal{T})=s(\mathcal{A})$. The proof is completed by combining the result of Proposition \ref{Spectrum} with \eqref{Exp.cond} to yield $\omega_0(\calT)<0$. 
\end{proof}

Returning to the $\mathbf{aBCS}$ we see that in view of the above results it is now readily verified that it admits a unique positive mild solution $z\in \bm{C}(\R_+;{X})$ for any $x\in X$, given by
\begin{equation*}
{z}(t) =\calT(t)x,\quad t\in\R_+,
\end{equation*}
and that there exist constants $J, \alpha>0$ such that for all $t\in\R_+$ and any $x\in X$,
\begin{equation*}
\| z(t)\|_{X}\le J e^{-\alpha t}\| x\|_{ X}.
\end{equation*}

\section{Applications}\label{Sec:3}

We now consider the use of the results of the preceding section in determining the well-posedness and exponential stability of a number of applications and physical system models.

\subsection{Boundary delay equations}

The aim of this subsection is to show how the class of abstract boundary delay equations introduced in \cite{GANTOUH2026106307} fits into the framework of Section \ref{Sec:1} so that the results of Section \ref{Sec:2} can be applied. Following the notational framework of Section \ref{Sec:1.2}, we consider an abstract dynamical system described by
\begin{equation}\label{DEq}
\dot{z}(t)=Az (t),\quad y(t)=K z(t),\quad t>0,
\end{equation}
with boundary and initial conditions
\begin{equation}\label{DEqx}
Gz(t)= L(y_t),\quad t>0,\quad z(0)=x,\quad y_0=\varphi,
\end{equation}
wherein
\begin{itemize}[leftmargin=*,align=left,labelwidth=\parindent]
\item $\varphi\in \bm{L}_p(-r,0;U)\coloneqq Y_p$ for $0<r<\infty$, $1\le p<\infty$;
\item $L:\bm{W}^{1}_p(-r,0;U)\to U$ is a bounded linear operator; and
\item $y:[-r,+\infty)\to U$ and $y_t:[-r,0]\to U$ as defined by $y_t(\theta)=y(t+\theta)$.
\end{itemize}

For the purposes of proving well-posedness and stability results it will be required to reformulate \eqref{DEq}, \eqref{DEqx} in the form of the $\mathbf{aBCS}$. We observe, first of all, that $y_t=y_t\left(\theta\right)$ satisfies the following initial/boundary-value problem for the transport equation:
\begin{equation}\label{transport-eq}
\dot{y}_t= \partial_\theta y_t,\quad y_t(0)= y(t), \quad t> 0,\quad y_0=\varphi.
\end{equation}
The reformulation is completed by combining \eqref{transport-eq} with \eqref{DEq}, \eqref{DEqx} to give
\begin{equation*}
\dot{z}(t)=A z(t),\quad Gz(t)= L(y_t), \quad \dot{y}_t= \partial_\theta y_t, \quad y_t(0)=Kz(t),\quad t>0,
\end{equation*}
with $U_1=U_2=U$. Define the operator
\begin{equation*}
Q\varphi\coloneqq \partial_\theta \varphi, \quad \bm{D}(Q)\coloneqq \ker\delta_0,
\end{equation*}
where $\delta_0:\bm{W}^{1}_p(-r,0;U)\to U$ is the Dirac measure concentrated at $0$. It is well known that $Q$ generates a left-shift semigroup $S\coloneqq (S(t))_{t\in\R_+}$ on $Y_p$ given by 
\begin{equation*}
(S(t)\varphi)(\theta)=\left\{\begin{split}
0&,&& -t\le \theta\le 0,\\
 \varphi(t+\theta)&,&& -r\le \theta\le -t,\end{split}\right.
\end{equation*}
for any $\varphi\in Y_p$ and a.e.\ $\theta\in [-r,0]$. Clearly, $S$ is positive and $\varrho(Q)=\C$. We further define the operators
\begin{equation*}
\Theta\coloneqq (-Q_{-1})\varepsilon_0, \quad \tilde{L}\coloneqq L\vert_{\bm{D}(Q)},
\end{equation*} 
where $(\varepsilon_\la y)(\theta)\coloneqq e^{\la \theta}y$ for $\la \in \C$, $\theta\in [-r,0]$, and any $y\in U$.

We are now ready to apply Theorem \ref{theorem-dynamic-boundary} to the system described by \eqref{DEq}, \eqref{DEqx} and obtain the well-posedness result which we state as a corollary.
\begin{corollary}\label{Cor1}
Let $X$, $U$ be Banach lattices and let Assumption \ref{Assp11} be satisfied. Identify the operators $\tilde{A}$, $B$, and $C$ of Definition \ref{regularity} as $A$, $G$, and $K$, respectively. Suppose that the triples $(\tilde{A},B, C)$, $(Q,\Theta,\tilde{L})$ with corresponding input/output maps $\F_{1}$ and $\F_2$, respectively, are positive $\bm{L}_p$-well-posed weakly regular triples with zero feedthrough in the sense of Definition \ref{regularity}. If the spectral radius $r(\F_1\F_2)<1$ or, equivalently, $r(\F_2\F_1)<1$, then the operator 
\begin{equation}\label{A-delay}
\calA\coloneqq\operatorname{diag}(A, \partial_\theta),\quad {\bm D}(\calA)\coloneqq \left\{x\in {\bm D}(A),~\varphi\in \bm{W}^{1}_p(-r,0;U): Gx=L(\varphi),~\varphi(0)=Kx\right\}
\end{equation}
generates a positive $C_0$-semigroup $\calT\coloneqq (\calT(t))_{t\in\R_+}$ on $X\times Y_p$. Moreover, there exists for any $x\in X$ and $\varphi\in Y_p$ a unique positive mild solution $z\in \bm{C}(\R_+,X)$ of the system \eqref{DEq}, \eqref{DEqx} which can be represented in the form
\begin{align*}
z(t,x,\varphi)&=T(t)x+ \int_{0}^{t}T_{-1}(t-s)B\tilde{L}_{\Lambda} S(s)\varphi ds
\\&\quad+ \int_{0}^{t}T_{-1}(t-s)B\tilde{L}_{\Lambda} \int_{0}^{s}S_{-1}(s-\sigma)\Theta C_\Lambda z(\sigma,x,\varphi)d\sigma ds,\quad t\in\R_+,
\end{align*}
where $\tilde{L}_{\Lambda}$, $C_\Lambda$ are the Yoshida extensions of $\tilde{L}$ and $C$ with respect to $Q$ and $\tilde{A}$, respectively. 
\end{corollary}

It is worth noting that Corollary \ref{Cor1} is a slightly generalized statement of \cite[Theorem 4.1]{GANTOUH2026106307}. Moreover, in the following result, we recover the spectral characterization of exponential stability given in \cite[Theorem 4.2]{GANTOUH2026106307}.
\begin{corollary}
Let the assumptions of Corollary \ref{Cor1} be satisfied. If 
\begin{equation*}
s(\tilde{A})<0,\quad r\left(KD_0L(\varepsilon_0)\right)<1,
\end{equation*}
then there exist constants $N_{x,\varphi}, \alpha>0$ such that
\begin{equation*}
\Vert z(t)\Vert_{X}+\Vert y_t\Vert_{Y_p}\le N_{x,\varphi} e^{-\alpha t}\left(\Vert x\Vert_{X}+\Vert\varphi\Vert_{Y_p}\right),
\end{equation*}
for all $t\in\R_+$ and any $x\in {\bm D}(A)$, $\varphi\in \bm{W}^{1}_p(-r,0;U)$ satisfying the coupling conditions $Gx=L(\varphi)$ and $ \varphi(0)=Kx$.
\end{corollary}
\begin{proof}
The proof follows from the result of Theorem \ref{stability-result} and the fact that $s(Q)=-\infty$.
\end{proof}

In closing this application example, we note that in the setting of Hilbert lattices, application of Proposition \ref{sharp-wellposedness-cond} provides the following sharp well-posedness result for the system \eqref{DEq}, \eqref{DEqx}.
\begin{corollary}
Let $X$, $U$ be Hilbert lattices and let Assumption \ref{Assp11} be satisfied. Identify the operators $\tilde{A}$, $B$, and $C$ of Definition \ref{regularity} as $A$, $G$, and $K$, respectively. Suppose that the triples $(\tilde{A},B, C)$, $(Q,\Theta,\tilde{L})$ are positive $\bm{L}_2$-well-posed weakly regular triples with zero feedthrough in the sense of Definition \ref{regularity}. Then the system operator $\calA$ defined by \eqref{A-delay} generates a positive $C_0$-semigroup $\calT\coloneqq (\calT(t))_{t\in\R_+}$ on $X\times Y_2$ if and only if $r\left(KD_{\lambda_0}L(\varepsilon_{\lambda_0})\right)<1$ for some $\lambda_0>\omega_0(T)$.
\end{corollary}

\subsection{Transport equations on a cyclic network with time-delayed vertex conditions}

In this subsection we consider a network of $m\ge 2$ nonconservative transport equations
\begin{equation}\label{transport-equation}
\dot{u}_j(t,x)= -v_j(x)\partial_{x}u_j(t,x),\quad 1\le j\le m,
\end{equation}
in a cyclic arrangement, by which we mean a network consisting of one vertex and $m$ loops attached to the vertex. The conditions that specify the coupling at the vertex are
\begin{equation}\label{transport-equation2}
\gamma_jv_j(0)u_j(t,0)=\displaystyle\sum_{k=1}^{m}\sigma_{jk}\xi_kv_k(\ell_k)u_k(t-r_k,\ell_k),\quad 1\le j\le m,
\end{equation}
and the initial conditions are
\begin{equation}\label{transport-equation3}
u_j(0,x)= f_j(x),\quad u_j(\theta,\ell_j)=\varphi_j(\theta),\quad 1\le j\le m,
\end{equation}
given functions $f_j$, $\varphi_j$. Above ${u}_j(t,x)$ represents the flow of particles on the $j$-th cycle for $(t,x)\in \R_+\times[0,\ell_j]$, $\ell_j>0$ being the length of each cycle. The functions $v_j(\cdot)$ are the particle-flow velocities, the parameters $\gamma_j$, $\xi_j$ are the absorption and amplification coefficients at the points $x=0$ and $x=\ell_j$, respectively, the $\sigma_{jk}\in \R$ represent the coupling coefficients, and $\theta\in[-r_j,0]$, the $r_j\in \R_+$ being inter-cycle transmission time delays in the network.

The problem of exponential stability for the initial/boundary-value problem \eqref{transport-equation}--\eqref{transport-equation3} in the time-varying (i.e.\ nonautonomous) nonconservative case, with partially internal local time-dependent parameters, was addressed in \cite{MR3624395} in the case when time delays are absent. Here, via \eqref{transport-equation2}, we account for both nonconservativeness as well as time delays in the transmission of flows at the vertex and, additionally, assume the flow velocities to be spatially variable (i.e.\ nonuniform).

To begin our study, we consider the Banach spaces
\begin{equation*}
X_1=\prod_{j=1}^m {\bm L}_p(0,\ell_j;\R),\quad X_2=\prod_{j=1}^m {\bm L}_p(-r_j,0;\R)
\end{equation*}
with respective norms
\begin{equation*}
\Vert f\Vert_{X_1}=\sum_{j=1}^{m}\Vert f_j\Vert_{{\bm L}_p(0,\ell_j;\R)},\quad \Vert \varphi\Vert_{X_2}=\sum_{j=1}^{m}\Vert \varphi_j\Vert_{{\bm L}_p(-r_j,0;\R)}.
\end{equation*}
We introduce a history segment
\begin{equation*}\label{segment1}
[-r_j,0]\ni \theta\mapsto \tilde{u}_j^t(\theta)\coloneqq \tilde{u}_j(t+\theta), \quad 1\le j\le m,
\end{equation*}
wherein the tilde indicates that the flows are to be evaluated at the points $x=\ell_j$, i.e., $\tilde{u}_j(t)=u_j(t,\ell_j)$ for every $t\in\R_+$. In particular, for $1\le j\le m$, the function $t\mapsto\tilde{u}_j^t(t)$ satisfies \eqref{transport-eq}. The initial/boundary-value problem \eqref{transport-equation}--\eqref{transport-equation3} can be formulated as an abstract system on $X=X_1\times X_2$, the product space formed by the state space,
\begin{equation*}
\dot{y}(t)=\calA y(t),\quad t> 0,\quad y(0)=((f_j)_{j=1}^m, (\varphi_j)_{j=1}^m)^\top,
\end{equation*} 
where $y(t)=((u_j(t,\cdot))_{j=1}^m,(\tilde{u}_j^t)_{j=1}^m)^\top$ and the system operator $\calA$ on $X$ is defined by 
\begin{equation}\label{cauchy-pb-ope}
\begin{split}
\calA &\coloneqq \operatorname{diag}(- V(\cdot)\partial_x, \partial_\theta),\\
{\bm D}(\calA) & \coloneqq \left\{\begin{gathered}(f,\varphi)^\top\in \prod_{j=1}^m {\bm W}_p^1(0,\ell_j;\R)\times \prod_{j=1}^m {\bm W}_p^1(-r,0;\R):\\ \Gamma V(0)f(0)=(\varphi_j(-r_j))_{j=1}^m,~
\varphi(0)=\mathbb{M}\Xi \operatorname{diag}(v_k(\ell_{k}))_{k=1}^m (f_k(\ell_k))_{k=1}^m 
\end{gathered}\right\}, 
\end{split}
\end{equation}
wherein $V(\cdot)\coloneqq \operatorname{diag}(v_j(\cdot))_{j=1}^m$, $\Gamma\coloneqq \operatorname{diag}(\gamma_j)_{j=1}^m$, $\mathbb{M}\coloneqq (\sigma_{jk})_{j,j=1}^m$, and $\Xi\coloneqq \operatorname{diag}(\xi_j)_{j=1}^m$.

We are now in a position to prove the main result of this subsection.
\begin{theorem}\label{stab-trp-sys}
For $1\le j\le m$, suppose $v_j\in {\bm L}_\infty(0,\ell_j)$ with $v_j(x)>v_0>0$ for a.e.\ $x\in [0,\ell_j]$. Suppose further that $\gamma_j,\xi_j,\sigma_{jk}>0$ for $1\le j,k\le m$. If
\begin{equation}\label{stability-condition}
r(\mathbb{M}\Xi \operatorname{diag}(v_k(\ell_{k}))_{k=1}^m V^{-1}(0)\Gamma^{-1})<1,
\end{equation}
then there exist constants $J,\alpha>0$ such that for all $t\in\R_+$ and any $f\in X_1$, $\varphi\in X_2$,
\begin{equation}\label{solution-decay}
\Vert u(t)\Vert_{X_1}+\Vert \tilde{u}^t\Vert_{X_2}\le J e^{-\alpha t}\left(\Vert f\Vert_{X_1}+\Vert \varphi\Vert_{X_2}\right).
\end{equation}
\end{theorem}

\begin{proof}
We apply Corollary \ref{Cor2} with $U_1=U_2=\R^m$. To this end, let us define
\begin{equation*}
A_1 f\coloneqq - V(\cdot)\partial_x f, \quad A_2 \varphi\coloneqq \partial_\theta \varphi
\end{equation*}
for any $f\in {\bm D}(A_1)\coloneqq \prod_{j=1}^m {\bm W}_p^1(0,\ell_j;\R)$ and $\varphi\in {\bm D}(A_2)\coloneqq \prod_{j=1}^m{\bm W}_p^1(-r_j,0;\R)$. Let us further define the boundary operators
\begin{gather*}
G_1f\coloneqq \Gamma V(0)f(0),\quad K_1f\coloneqq \mathbb{M}\Xi\operatorname{diag}(v_k(\ell_{k}))_{k=1}^m (f_k(\ell_k))_{k=1}^m,\\
G_2\varphi\coloneqq \varphi(0),\quad K_2 \varphi\coloneqq (\varphi_j(-r_j))_{j=1}^m
\end{gather*} 
for any $f\in {\bm D}(A_1)$ and $\varphi \in {\bm D}(A_2)$. With these definitions we next show that the system operator $\mathcal{A}$ defined by \eqref{cauchy-pb-ope}, now with domain 
\begin{equation*}
\calA \coloneqq \operatorname{diag}(A_1,A_2),\quad {\bm D}(\calA)\coloneqq \left\{(f, \varphi)^\top\in {\bm D}(A_1)\times {\bm D}(A_2): G_1 f=K_2 \varphi,~ G_2 \varphi=K_1 f \right\},
\end{equation*}
generates a positive $C_0$-semigroup on $X$. By assumption, $v_j(\cdot)>0$ for $1\le j\le m$; hence, for $1\le j\le m$, the function $[0,\ell_j]\ni x\mapsto\tau_j(x)\coloneqq \int_{0}^{x}\frac{1}{v_j(x)}dx$ is invertible for $x\in [0,\ell_j]$. Then, by a generation argument similar to the one found in \cite{MR2328116}, it can be shown that the operator restriction $\tilde{A}_1\coloneqq A_1\vert_{\ker G_1}$ generates a positive $C_0$-semigroup $T_1\coloneqq (T_1(t))_{t\in\R_+}$ on $X_1$ given by 
\begin{equation*}
(T_1(t)f)_j(x)=\left\{\begin{split}
f_j(\t_j^{-1}(\t_j(x)-t))&,&& t \leq \tau_j(x),&& 1\le j\le m,\\
0&,&&\text{otherwise},&& 1\le j\le m,
\end{split}\right.
\end{equation*}
for any $f\in X_1$ and a.e.\ $x\in [0,\ell_j]$. Since $\gamma_j,v_j(\cdot)>0$ for $1\le j\le m$, it follows that the operator $G_1$ is onto. Moreover, formal calculation shows that the Dirichlet operator associated with the pair $(A_1,G_1)$ is given by 
\begin{equation*}
D_\la^1 \mathsf{a}={\rm diag}(e^{-\la\tau_j(\cdot)})_{j=1}^m V^{-1}(0) \Gamma^{-1} \mathsf{a}
\end{equation*}
for all $\la\in \C$ and any $\mathsf{a}\in \R^m$. We put $B_1\coloneqq -(\tilde{A}_1)_{-1}D_0^1\in\calB(\R^m,(X_1)_{-1}^{\tilde{A}_1})$. Using Laplace transformation techniques, it is then easily shown that for all $t\in\R_+$ the input map corresponding to the pair $(\tilde{A}_1,B_1)$ is given by 
\begin{equation*}
(\Phi^{1}_t g)_j(x)=\left\{\begin{split}
\frac{g_j(t-\tau_j(x))}{\gamma_j v_j(0)}&,&& t\ge \tau_j(x),&& 1\le j\le m,\\
 0&,&&\text{otherwise},&& 1\le j\le m,
\end{split}\right.
\end{equation*}
for any $g\in {\bm L}_p(\R_+,\R^m)$ and a.e.\ $x\in [0,\ell_j]$. Clearly, $\Phi^{1}_t$ is positive for every $t\in\R_+$, and therefore $B_1$ is positive (cf.\ \cite[Lemma 2.2]{elgantouh2025admissibility}). Moreover, we have 
\begin{equation*}
\Vert \Phi^{1}_t g\Vert_{X_1}\le \frac{1}{v_0 \underline{\gamma}}\Vert g\Vert_{{\bm L}_p(0,t;\R^m)}
\end{equation*}
for all $t\in\R_+$ and any $g\in {\bm L}_{p,+}(\R_+,\R^m)$, where $\underline{\gamma}\coloneqq \inf_{1\le j\le m} \gamma_j$. Thus the pair $(\tilde{A}_1,B_1)$ is positive $\bm{L}_p$-admissible. Next set $C_1\coloneqq K_1\vert_{{\bm D}(\tilde{A}_1)}$. Choosing $\alpha\ge \sup_{1\le j\le m}\tau_j(\ell_j)$, we obtain that for any $f\in {\bm D}_+(\tilde{A}_1)$,
\begin{equation*}
\int_{0}^{\alpha} \Vert C_1 T_1(t)f\Vert^p_{\R^m} dt \le (\Vert\mathbb{M}\Vert_{\R^m}\bar{\xi}\bar{v})^p\Vert f\Vert_{X_1}^p,
\end{equation*}
where $\bar{v}\coloneqq \sup_{1\le j\le m}\Vert v_j\Vert_\infty$ and $\bar{\xi}\coloneqq \sup_{1\le j\le m} \xi_j$. Hence, the pair $(C_1,\tilde{A}_1)$ is positive $\bm{L}_p$-admissible. Now we define the operator
\begin{equation*}
(\mathbb{F}_1g)(t)\coloneqq K_1\Phi^{1}_tg
\end{equation*}
for all $t\in[0,\alpha]$, $\alpha\ge \sup_{1\le j\le m}\tau_j(\ell_j)$, and any $g\in \mathring{\bm{W}}^{1}_{p,+}(0,\alpha;\R^m)$. Then,
\begin{equation}\label{input-output-L}
(\mathbb{F}_1g)_j(t)= \left\{\begin{split}
\sum_{k=1}^{m} \frac{\sigma_{jk}\xi_kv_k(\ell_k)}{\gamma_k v_k(0)}g_{k}(t-\tau_{k}(\ell_{k}))&,&&t\ge \tau_k(\ell_k),&& 1\le j\le m,\\
0&,&&\text{otherwise},&& 1\le j\le m.
\end{split}\right.
\end{equation}
Hence, for any $g\in \mathring{\bm{W}}^{1}_{p,+}(0,\alpha;\R^m)$,
\begin{equation*}
\int_{0}^{\alpha} \Vert (\mathbb{F}_1g)(t)\Vert^p_{\R^m} dt\le \left(\frac{\Vert\mathbb{M}\Vert_{\R^m}\bar{\xi}\bar{v}}{v_0 \underline{\gamma}}\right)^p \Vert g\Vert_{L^p(0,\alpha;\R^m)}^p,
\end{equation*}
where we have applied the Hölder inequality, and it follows that the triple $(\tilde{A}_1,B_1,C_1)$ is a positive $\bm{L}_p$-well-posed triple on $(\R^m,X_1,\R^m)$. Furthermore, the transfer function of $(\tilde{A}_1,B_1,C_1)$ is explicitly given by 
\begin{equation*}
\mathbf{H}_1(\la)=K_1 D_\la^1= \mathbb{M}\Xi{\rm diag}(v_k(\ell_{k}))_{k=1}^m {\rm diag}(e^{-\la \tau_j(\ell_j)})_{j=1}^mV^{-1}(0) \Gamma^{-1}
\end{equation*}
and, therefore,
\begin{equation*}
\lim_{\la \to +\infty} \Vert \mathbf{H}_1(\la)\Vert_{\calB(\R^m)}=0.
\end{equation*}
Hence, the triple $(\tilde{A}_1,B_1,C_1)$ is a positive $\bm{L}_p$-well-posed uniformly regular triple with zero feedthrough in the sense of Definition \ref{regularity}. Let now $\tilde{A}_2$ be the restriction of $A_2$ to $\ker G_2$. Then $\tilde{A}_2$ generates a left-shift semigroup $T_2\coloneqq (T_2(t))_{t\in\R_+}$ on $X_2$ given by
\begin{equation*}
(T_2(t)\varphi)_j(\theta)=\varphi_j(t+\theta)\1_{-r_j\le\theta\le-t},\quad 1\le j\le m,
\end{equation*}
for any $\varphi\in X_2$ and a.e.\ $\theta\in [-r_j,0]$, with $\1_{-r_j\le\theta\le-t}$ the indicator function which is $1$ for $-r_j\le\theta\le-t$. Moreover, the operator $G_2$ being onto, a direct calculation yields that the Dirichlet operator associated with the pair $(A_2,G_2)$ is given by
\begin{equation*}
(D_\la^2 \mathsf{a})_j(\theta)=e^{\la \theta}\mathsf{a}_j,\quad 1\le j\le m,
\end{equation*}
for all $\la \in \C$, any $\mathsf{a}\in \R^m$, and a.e.\ $\theta\in [-r_j,0]$. In particular, choosing
\begin{equation*}
B_2\coloneqq -(\tilde{A}_2)_{-1}D_0^2\in\calB(\R^m,(X_2)_{-1}^{\tilde{A}_2}),\quad C_2\coloneqq K_2\vert_{{\bm D}(\tilde{A}_2)},
\end{equation*}
we obtain, by essentially the same arguments as above, that the triple $(\tilde{A}_2,B_2,C_2)$ is a positive $\bm{L}_p$-well-posed uniformly regular triple on $(\R^m,X_2,\R^m)$ with zero feedthrough. Moreover, for all $t\in [0,\alpha]$, $\alpha\ge \sup_{1\le j\le m}r_j$, and any $h\in \mathring{\bm{W}}^{1}_{p,+}(0,\alpha;\R^m)$,
\begin{equation*}
(\mathbb{F}_2h)_j(t)= \left\{\begin{split}
h_j(t-r_j)&,&& t\ge r_j,&& 1\le j\le m,\\
 0&,&&\text{otherwise},&& 1\le j\le m.
\end{split}\right.
\end{equation*}
Note that $\mathbb{F}_2h \in \mathring{\bm{W}}^{1}_{p,+}(0,\alpha;\R^m)$ and, choosing $\alpha_0<\inf_{1\le j\le m}\tau_j(\ell_j)$, it follows from \eqref{input-output-L} that
\begin{equation*}
(I-\mathbb{F}_1\mathbb{F}_2)h=h
\end{equation*}
for any $h\in \mathring{\bm{W}}^{1}_{p,+}(0,\alpha_0;\R^m)$. In fact, by density, we obtain that
\begin{equation*}
I-\mathbb{F}_1\mathbb{F}_2=I\in\calB({\bm L}_p(0,\alpha_0;\R^m)),
\end{equation*} 
$I-\mathbb{F}_1\mathbb{F}_2$ is therefore invertible, and consequently the spectral radius $r(\mathbb{F}_1\mathbb{F}_2)<1$. Hence, we conclude from Theorem \ref{theorem-dynamic-boundary} that the system operator $\calA$ as defined by \eqref{cauchy-pb-ope} generates a positive $C_0$-semigroup $\calT\coloneqq (\calT(t))_{t\in\R_+}$ on $X$. Since $s(\tilde{A}_1)=s(\tilde{A}_2)=-\infty$, Corollary \ref{Cor2} now implies that \eqref{stability-condition} yields \eqref{solution-decay}. Also, as implied by the results \cite{MR1469440}, it follows that $s(\calA)=\omega_0(\calA)$. This completes the proof.
\end{proof}

\begin{remark}
It is interesting to note that Theorem \ref{stab-trp-sys} provides a spectral characterization of exponential stability for the initial/boundary-value problem \eqref{transport-equation}--\eqref{transport-equation3} that is independent of both the cycle lengths and time delays. A similar result has been established very recently in \cite{EZZL} for a variant of \eqref{transport-equation}--\eqref{transport-equation3}, using a different method.
\end{remark}

As an application of Theorem \ref{stab-trp-sys} we next impose a Kirchhoff type condition on the coupling coefficients $\sigma_{jk}$ by adjoining to \eqref{transport-equation}--\eqref{transport-equation3} the condition
\begin{equation}\label{Kirchhoff}
\sum_{k=1}^m \sigma_{jk}=1, \quad 1 \le j\le m.
\end{equation}

\begin{corollary}
Let the assumptions of Theorem \ref{stab-trp-sys} be satisfied and suppose that the condition \eqref{Kirchhoff} holds. If 
\begin{equation}\label{C2}
\max_{k}\frac{\xi_kv_k(\ell_k)}{\gamma_kv_k(0)}<1,
\end{equation}
then solutions to the initial/boundary-value problem \eqref{transport-equation}--\eqref{transport-equation3} decay to zero exponentially.
\end{corollary}
\begin{proof}
Let $d_k = \frac{\xi_kv_k(\ell_k)}{\gamma_kv_k(0)}$, $1 \le k\le m$. Then
\begin{equation*}
\mathbb{M}\Xi \, {\rm diag}(v_k(\ell_{k}))_{k=1}^m V^{-1}(0)\Gamma^{-1}=\mathbb{M}\, {\rm diag}(d_k )_{k=1}^m,
\end{equation*}
where $\mathbb{M}$ is by \eqref{Kirchhoff} a row-stochastic matrix whose entries are positive. Since the spectral radii $r\left( \mathbb{M} \, {\rm diag}(d_k) \right)$ have the property that 
\begin{equation*}
r\left( \mathbb{M} \, {\rm diag}(d_k) \right) 
\le \max_{1 \le j\le m} \sum_{k=1}^m \sigma_{jk} d_k,
\end{equation*}
we have, using \eqref{C2},
\begin{equation*}
\sum_{k=1}^m \sigma_{jk} d_k 
\le \max_k d_k \sum_{k=1}^m \sigma_{jk} 
= \max_k d_k < 1.
\end{equation*}
Hence, for $1 \le k\le m$,
\begin{equation*}
r\left( \mathbb{M} \, {\rm diag}(d_k) \right) < 1.
\end{equation*}
The result then follows from Theorem \ref{stab-trp-sys}.
\end{proof}

\subsection{Mixed heat-transport equations on a cyclic network}

In this subsection we consider a general network related to the special case depicted in Fig.\ \ref{fig01}, namely, a network of $m\ge 2$ mixed heat-transport equations
\begin{equation}\label{transport-heat}
\dot{u}_j(t,x)=\ka_j \partial_{xx}u_j(t,x)-\beta_ju_j(t,x),\quad
\dot{v}_j(t,x)=- c_j(x)\partial_{x}v_j(t,x),\quad 1 \le j\le m,
\end{equation}
together with conditions imposed at the junction point that specify the coupling at the vertex,
\begin{equation}\label{transport-heat-BC}
\begin{gathered}
u_i(t,0)=u_j(t,0),\quad-\sum_{j=1}^{m}\ka_j\partial_{x}u_j(t,0)=\sum_{j=1}^{m}c_j(\ell_j)v_j(t,\ell_j),\\
\ka_j\partial_{x}u_j(t,\alpha_j\ell_j)=0,\quad c_j(\alpha_j\ell_j)v_j(t,\alpha_j\ell_j)= u_j(t,\alpha_j\ell_j),\quad 1 \le i,j\le m,
\end{gathered}
\end{equation}
and the given initial conditions
\begin{equation}\label{transport-heat-IC}
u_j(x,0)=g_j^0(x),\quad v_j(x,0)=f^0_j(x),\quad 1 \le j\le m.
\end{equation}
Here ${u}_j(t,x)$ describes the temperature of a bar or wire on the $j$-th cycle for $(t,x)\in \R_+\times [0,\alpha_j\ell_j]$, $v_j(t,x)$ is the flow of particles on the $j$-th cycle for $(t,x)\in \R_+\times[\alpha_j\ell_j,\ell_j]$, the parameter $\ell_j>0$ again is the length of each cycle, the $\alpha_j\in (0,1)$ are some scaling parameters, the functions $c_j(\cdot)$ are the particle-flow velocities, and the parameters $\kappa_j$, $\beta_j$ are the thermal conductivity and absorption coefficients of the $j$-th cycle, respectively.

\begin{figure}[htb]
\centering
\includegraphics[width=0.4\linewidth]{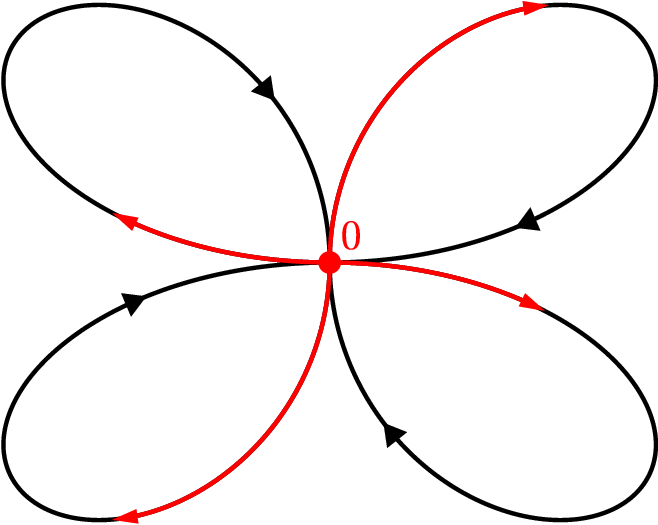}
\caption{Mixed heat-transport equations on a network consisting of four cycles: heat equations (${\color{red}\full}$) and transport equations (${\color{black}\full}$). \label{fig01}}
\end{figure}

In the following we investigate the problems of well-posedness and exponential stability of solutions of our initial/boundary-value problem \eqref{transport-heat}--\eqref{transport-heat-IC}. To do this we introduce the Banach spaces
\begin{equation*}
X_1\coloneqq \prod_{j=1}^m \bm{L}_1(0,\alpha_j\ell_j),\quad X_2\coloneqq \prod_{j=1}^m \bm{L}_1(\alpha_j\ell_j,\ell_j)
\end{equation*}
and define the operators $A_1$ and $A_2$ in $X_1$ and $X_2$, respectively, by
\begin{equation}\label{system.operators}
A_1g\coloneqq (\kappa_j\partial_{xx}g_j-\beta_jg_j)_{j=1}^m,\quad 
A_2f\coloneqq (-c_j(\cdot)\partial_{x}f_j)_{j=1}^m,
\end{equation}
with domains 
\begin{equation}\label{system.operators_dom}
\begin{split}
{\bm D}(A_1)&\coloneqq \left\{g\in \prod_{j=1}^m \bm{W}_1^2(0,\alpha_j\ell_j):g_i(0)=g_j(0),~\ka_j\partial_{x}g_j(\alpha_j\ell_j)=0,~1 \le i,j\le m\right\},\\
{\bm D}(A_2)&\coloneqq \prod_{j=1}^m \bm{W}_1^1(\alpha_j\ell_j,\ell_j).
\end{split}
\end{equation}
We also introduce the boundary coupling operators $G_k$, $K_k$, $k=1,2$: 
\begin{equation}\label{eq:1}
\begin{gathered}
G_1g\coloneqq -\sum_{j=1}^{m}\ka_j\partial_{x}g_j(0),\quad K_1g\coloneqq (g_j(\alpha_j\ell_j))_{j=1}^m,\\
 G_2f\coloneqq (c_j(\alpha_j\ell_j)f_j(\alpha_j\ell_j))_{j=1}^m,\quad K_2f\coloneqq \sum_{j=1}^{m}c_j(\ell_j)f_j(\ell_j),
\end{gathered}
\end{equation}
for any $g\in {\bm D}(A_1)$, $f\in {\bm D}(A_2)$. With these definitions in place we obtain the representation by the $\mathbf{aBCS}$ for our initial/boundary-value problem \eqref{transport-heat}-\eqref{transport-heat-IC} with $U_1=\R$ and $U_2=\R^m$, wherein
\begin{equation*}
 z_1(t)= (u_j(t,\cdot))_{j=1}^m,\quad z_2(t)= (v_j(t,\cdot))_{j=1}^m, \quad x_1= (g_j^0)_{j=1}^m, \quad x_2= (f_j^0)_{j=1}^m.
\end{equation*}

In order to apply our abstract results in Theorems \ref{Result1} and \ref{Result2} to this specific situation, we need some preliminary lemmas.
\begin{lemma}\label{heat-star-exp-stab}
Let $\ka_j,\beta_j>0$ for $1 \le j\le m$. Then the operator restriction
\begin{equation*}
\tilde{A}_1\coloneqq A_1\vert_{\ker G_1}
\end{equation*}
is densely defined and resolvent positive. For any $h\in (X_1)_+$, all $\lambda\ge 0$, and $\bar{\beta}:=\sup_{1\le j\le m}\beta_j$, the following estimate is valid:
\begin{equation}\label{invers-estimate}
\Vert R(\lambda,\tilde{A}_1)h\Vert_{X_1}\ge \frac{1}{\lambda+\bar{\beta}}\Vert h\Vert_{X_1}.
\end{equation}
Furthermore, $\tilde{A}_1$ generates an exponentially stable positive $C_0$-semigroup $T_1\coloneqq (T_1(t))_{t\ge 0}$ on $X_1$.
\end{lemma}
\begin{proof}
We begin by observing that the domain
\begin{equation*}
{\bm D}(\tilde{A}_1)\coloneqq\left\{\begin{gathered}g\in \prod_{j=1}^m \bm{W}_1^2(0,\alpha_j\ell_j):g_i(0)=g_j(0),~-\sum_{j=1}^{m}\ka_j\partial_{x}g_j(0)=0,\\ \ka_j\partial_{x}g_j(\alpha_j\ell_j)=0,~ 1 \le i,j\le m\end{gathered}\right\},
\end{equation*}
dense in $X_1$. For $h\in X_1$ and $g\in \bm{D}(\tilde{A}_1)$, let us consider the problem $-\tilde{A}_1g=h$. Then,
\begin{equation*}
\begin{gathered}
-\ka_j\partial_{xx}g+\beta_j g_j=h_j,\\ g_i(0)=g_j(0),\quad-\sum_{j=1}^{m}\ka_j\partial_{x}g_j(0)=0,\quad \ka_j\partial_{x}g_j(\alpha_j\ell_j)=0,\quad 1 \le i,j\le m,
\end{gathered}
\end{equation*}
and application of routine computations yields that the solutions $g_j=g_j(x)$ of the problem are given by
\begin{equation*}
g_j(x) = \sum_{k=1}^m \int_0^{\alpha_k \ell_k} \Gamma_{j,k}(x,s) h_k(s) ds, \quad 1 \le j\le m,
\end{equation*}
where 
\begin{equation*}
\Gamma_{j,k}(x,s) = \left\{
\begin{split}
\frac{ \cosh(\gamma_j (x-\alpha_j \ell_j)) \cosh(\gamma_k (s-\alpha_k \ell_k)) }{ \cosh(\gamma_j \alpha_j \ell_j) \cosh(\gamma_k \alpha_k \ell_k)N } &, && j \neq k,\\
\frac{\cosh(\gamma_j (x-\alpha_j \ell_j)) \cosh(\gamma_j (s-\alpha_j \ell_j)) }{ \cosh(\gamma_j \alpha_j \ell_j)^2 N} +\frac{ \cosh(\gamma_j (x-\alpha_j \ell_j)) \sinh(\gamma_j s) }{ \cosh(\gamma_j \alpha_j \ell_j) \kappa_j \gamma_j }\\+\frac{ \sinh(\gamma_j (x-s)) }{ \kappa_j \gamma_j } \1_{s \geq x}&, && j = k,
\end{split}\right.
\end{equation*}
with $\gamma_j \coloneqq \sqrt{{\beta_j}/{\kappa_j}}$ and $N = \sum_{k=1}^m \kappa_k \gamma_k \tanh(\gamma_k \alpha_k \ell_k)$. 
Clearly, $\Gamma_{j,k}(x,s)\ge 0$ for $j \neq k$. We show that $\Gamma_{j,k}(x,s)\ge 0$ for $j =k$ as well. Let us set
\begin{equation*}
\Gamma_{j,k}(x,s) = \Gamma^1_{j,k}(x,s) + \Gamma^2_{j,k}(x,s),
\end{equation*}
where
\begin{align*}
\Gamma^1_{j,k}(x,s) &\coloneqq \frac{\cosh(\gamma_j (x-\alpha_j \ell_j)) \cosh(\gamma_j (s-\alpha_j \ell_j)) }{ \cosh(\gamma_j \alpha_j \ell_j)^2 N}, \\
\Gamma^2_{j,k}(x,s) &\coloneqq \frac{ \cosh(\gamma_j (x-\alpha_j \ell_j)) \sinh(\gamma_j s) }{ \cosh(\gamma_j \alpha_j \ell_j) \kappa_j \gamma_j }+\frac{ \sinh(\gamma_j (x-s)) }{ \kappa_j \gamma_j } \1_{s \geq x}.
\end{align*}
Then we obviously have $\Gamma^1_{j,j}(x,s)\ge 0$. Now, using the elementary identities
\begin{align*}
\sinh(\gamma_j (x-s))& = \sinh(\gamma_j x) \cosh(\gamma_j s)-\cosh(\gamma_j x) \sinh(\gamma_j s),\\
\cosh(\gamma_j (x-\alpha_j \ell_j))& = \cosh(\gamma_j x) \cosh(\gamma_j \alpha_j \ell_j)-\sinh(\gamma_j x) \sinh(\gamma_j \alpha_j \ell_j),
\end{align*}
we obtain
\begin{align*}
\Gamma^2_{j,j} (x,s)&= \frac{1}{ \kappa_j \gamma_j}\sinh(\gamma_j x) \left( \cosh(\gamma_j s)-\tanh(\gamma_j \alpha_j \ell_j) \sinh(\gamma_j s) \right)\\
&=\frac{1}{ \kappa_j \gamma_j}\sinh(\gamma_j x)\frac{ \cosh(\gamma_j(s-\alpha_j \ell_j)) }{ \cosh(\gamma_j \alpha_j \ell_j) }\\
&\geq \frac{1}{ \kappa_j \gamma_j}\frac{\sinh(\gamma_j x)}{ \cosh(\gamma_j \alpha_j \ell_j)}.
\end{align*}
Since $\cosh(\gamma_j(s-\alpha_j \ell_j))\ge 1$ for $s\in [0,\alpha_j\ell_j]$, it follows that $\Gamma^2_{j,j} (x,s)\ge 0$ and we conclude that $\Gamma_{j,j}(x,s) = \Gamma^1_{j,j}(x,s) + \Gamma^2_{j,j}(x,s)\ge 0$. Hence, $0\in \varrho(\tilde{A}_1)$, implying that $\tilde{A}_1$ is boundedly invertible (and is therefore closed and densely defined), and $-\tilde{A}_1^{-1}\ge 0$. Since $\tilde{A}_1$ is coercive for all $\lambda\ge 0$, it follows that $[0,\infty)\subset \varrho(\tilde{A}_1)$. To proceed further we let $R(\lambda,\tilde{A}_1)h\eqqcolon g$ be the positive solution of $(\lambda I-\tilde{A}_1)g=h$ for $h\in (X_1)_+$ and $\lambda\ge 0$. Then the $g_j=g_j(x)$ satisfy
\begin{equation*}
(\lambda+\beta_j)g_j-\kappa_j\partial_{xx}g_j=h_j, \quad 1 \le j\le m.
\end{equation*}
Integrating with respect to $x\in (0,\alpha_j\ell_j)$ and using the additivity of the ${\bm L}_1$-norm on $(X_1)_+$, we arrive at (recall $g\in \bm{D}(\tilde{A}_1)$)
\begin{equation*}
\sum_{j=1}^{m}(\lambda+\beta_j)\int_{0}^{\alpha_j\ell_j} g_j(x)dx=\Vert h\Vert_{X_1}, 
\end{equation*}
from which we obtain
\begin{equation*}
\Vert g\Vert_{X_1}\ge \frac{1}{\lambda+\bar{\beta}}\Vert h\Vert_{X_1}
\end{equation*}
with $\bar{\beta}:=\sup_{j}\beta_j$. Therefore, the operator $\tilde{A}_1$ is densely defined, resolvent positive, and satisfies the estimate \eqref{invers-estimate}. Moreover, according to \cite[Theorem 2.5]{MR872810}, $\tilde{A}_1$ generates a positive $C_0$-semigroup $T_1\coloneqq (T_1(t))_{t\ge 0}$ on $X_1$. Since $0\in \varrho(\tilde{A}_1)$ and $(-\tilde{A}_1)^{-1}\ge 0$, it follows from \cite[Theorem C-IV.1.1 and Corollary C-IV.1.4]{MR839450} that $\omega_0(T_1)<0$. With this the proof is complete.
\end{proof}

\begin{remark}
We note that the operator restriction $\tilde{A}_1\coloneqq A_1\vert_{\ker G_1}$ in Lemma \ref{heat-star-exp-stab} could be associated with the abstract Cauchy problem for a compact star-network setup of heat equations. In particular, Lemma \ref{heat-star-exp-stab} would establish the well-posedness, positivity, and exponential stability of the corresponding solutions.
\end{remark}

\begin{lemma}\label{Dirichlet-heat-lemma}
Let the assumption of Lemma \ref{heat-star-exp-stab} be satisfied. Then the operator $G_1$ defined in \eqref{eq:1} is onto from ${\bm D}(A_1)$ to $\R$. Moreover, for all $\lambda>-\bar{\beta}$, any $d\in \R$, and a.e.\ $x\in [0,\alpha_j\ell_{j}]$, the Dirichlet operator associated with the pair $(A_1,G_1)$ is given by
\begin{equation}\label{Dirichlet-heat}
(D_\la^1 d)_j(x)= \frac{\cosh(\gamma_j (x-\alpha_j \ell_j))}{\cosh(\gamma_j \alpha_j \ell_j)}\sum_{k=1}^m \frac{d}{\kappa_k \gamma_k \tanh(\gamma_k \alpha_k \ell_k)},\quad 1 \le j\le m,
\end{equation}
with $\gamma_j\coloneqq\sqrt{{(\beta_j+\lambda)}/{\kappa_j}}$ and is positive.
\end{lemma}

\begin{proof}
Recall the definitions of $A_1$, $G_1$ in \eqref{system.operators}--\eqref{eq:1}. For $g\in \bm{D}({A}_1)$ and $d\in \R$, we first consider the following generalized elliptic spectral problem:
\begin{equation*}
(\lambda-A_1)g=0,\quad G_1 g=d,
\end{equation*}
Then, for fixed $\lambda > -\bar{\beta}$, the $g_j=g_j(x)$ satisfy
\begin{equation*}
-\kappa_j \partial_{xx} g_j + (\beta_j+\lambda) g_j = 0,\quad 1 \le j\le m.
\end{equation*}
Put
\begin{equation}\label{expsubd123}
 \gamma_j \coloneqq \sqrt{\frac{\beta_j + \lambda}{\kappa_j}} > 0,\quad 1 \le j\le m.
\end{equation}
It is easily verified that the general solution is given by
\begin{equation*}
g_j(x) = a_j e^{\gamma_j x} + b_j e^{-\gamma_j x},\quad 1 \le j\le m.
\end{equation*}
Using the conditions $ \kappa_j \partial_x g_j(\alpha_j \ell_j) = 0$, we obtain
\begin{equation}\label{expsub123}
a_j = b_j e^{-2\gamma_j \alpha_j \ell_j},\quad 1 \le j\le m.
\end{equation}
The vertex continuity condition $g_i(0)=g_j(0)$, $1 \le j\le m$, implies $a_i + b_i =a_j + b_j= c$, some constant $c\in \R$, and substitution of \eqref{expsub123} gives
\begin{equation*}
a_j=\frac{ce^{-2\gamma_j \alpha_j \ell_j}}{1+e^{-2\gamma_j \alpha_j \ell_j}},\quad b_j=\frac{c}{1+e^{-2\gamma_j \alpha_j \ell_j}},\quad 1 \le j\le m.
\end{equation*}
Thus,
\begin{equation}\label{eqsol1123}
g_j(x) = \frac{c}{1 + e^{-2\gamma_j \alpha_j \ell_j}} \left( e^{\gamma_j x - 2\gamma_j \alpha_j \ell_j} + e^{-\gamma_j x} \right)
=c\frac{\cosh(\gamma_j (x - \alpha_j \ell_j))}{\cosh(\gamma_j \alpha_j \ell_j)},\quad 1 \le j\le m.
\end{equation}
In order to apply the remaining condition $G_1 g = d$ to the solution obtained in \eqref{eqsol1123}, we first compute
\begin{equation*}
\partial_x g_j(0) = c\gamma_j\frac{\sinh(-\gamma_j \alpha_j \ell_j)}{\cosh(\gamma_j \alpha_j \ell_j)}=-c\gamma_j \tanh(\gamma_j \alpha_j \ell_j),\quad 1 \le j\le m.
\end{equation*}
Then,
\begin{equation*}
G_1 g = -\sum_{j=1}^m \kappa_j \partial_x g_j(0) = c \sum_{j=1}^m \kappa_j \gamma_j \tanh(\gamma_j \alpha_j \ell_j),
\end{equation*}
which yields that
\begin{equation*}
c=\sum_{j=1}^m\frac{d}{ \kappa_j \gamma_j \tanh(\gamma_j \alpha_j \ell_j)}
\end{equation*}
on account of $G_1g=d$. Substituting this into \eqref{eqsol1123}, we obtain
\begin{equation*}
g_j(x) = \frac{\cosh(\gamma_j (x-\alpha_j \ell_j))}{\cosh(\gamma_j \alpha_j \ell_j)}\sum_{k=1}^m\frac{d}{ \kappa_k \gamma_k \tanh(\gamma_k \alpha_k \ell_k)},\quad 1 \le j\le m.
\end{equation*}
This demonstrates that for any $d \in \R$, there exists a $g \in \bm{D}(A_1)$ such that $(\lambda - A_1)g = 0$ and $G_1 g = d$, proving that $G_1$ is onto. By definition, the Dirichlet operator associated with the pair $(A_1,G_1)$ is given by $D_\lambda^1 d = g$, which verifies \eqref{Dirichlet-heat}. The fact that, under the condition $d\ge 0$, for all $\lambda>-\bar{\beta}$ the operator $D_\la^1$ is positive is then obvious from \eqref{Dirichlet-heat}. This completes the proof of the lemma.
\end{proof}

The main result of this subsection presented below addresses the exponential stability of the initial/boundary-value problem \eqref{transport-heat}-\eqref{transport-heat-IC}.
\begin{theorem}
For $1 \le j\le m$, suppose $c_j\in \bm{L}_{\infty}(\alpha_j\ell_{j}, \ell_{j})$ with $c_j(x) >c_0> 0$ for a.e.\ $x\in [\alpha_j\ell_j, \ell_j]$. Suppose further that $\kappa_j, \beta_j > 0$ for $1 \le j\le m$. Then, for any $g,f\in X = X_1\times X_2\coloneqq \prod_{j=1}^m \bm{L}_1(0,\alpha_j\ell_j) \times \prod_{j=1}^m \bm{L}_1(\alpha_j\ell_j,\ell_j)$, there exists a unique positive mild solution $z \in \bm{C}(\R_+;X)$ of the $\mathbf{aBCS}$, where the operators are defined by \eqref{system.operators}--\eqref{eq:1}. If
\begin{equation}\label{stability-cond}
\left( \sum_{k=1}^m \frac{1}{\sqrt{\kappa_k \beta_k} \tanh(\alpha_k \ell_k \sqrt{\frac{\beta_k}{\kappa_k}})} \right)
\left( \sum_{j=1}^m \frac{1}{\cosh(\alpha_j \ell_j \sqrt{\frac{\beta_j}{\kappa_j}})} \frac{c_j(\ell_j)}{c_j(\alpha_j \ell_j)} \right) < 1
\end{equation} 
then solutions to the initial/boundary-value problem \eqref{transport-heat}-\eqref{transport-heat-IC} decay to zero exponentially.
\end{theorem}
\begin{proof}
We can apply Theorem \ref{Result1} and obtain the first statement on the well-posedness. Proceeding in the usual way, let us first define the system operator $\mathcal{A}$ by
\begin{equation}\label{eqsysop123}
\calA \coloneqq {\rm diag}(A_1,A_2),\quad
 {\bm D}(\calA)\coloneqq \left\{(f, \varphi)^\top\in {\bm D}(A_1)\times {\bm D}(A_2): G_1 f=K_2 \varphi,~ G_2 \varphi=K_1 f \right\},
\end{equation}
where the operators $A_k$, $G_k$, $K_k$, $k=1,2$, are defined by \eqref{system.operators}--\eqref{eq:1}. For $k=1,2$, let $\tilde{A}_k:{\bm D}(\tilde{A}_k)\subseteq X_k\to X_k$ be the operator restriction defined by $\tilde{A}_k\coloneqq A_k\vert_{\ker G_k}$. By Lemma \ref{heat-star-exp-stab}, $\tilde{A}_1$ is densely defined, resolvent positive, and satisfies the estimate \eqref{invers-estimate} for any $x\in X_+$ and all $\lambda\ge 0$. By Lemma \ref{Dirichlet-heat-lemma}, $G_1$ is onto and the Dirichlet operator associated with the pair $(A_1,G_1)$, $D_\la^1$, is positive for all $\lambda>-\bar{\beta}$. For $\tilde{A}_2$, an argument analogous to that in \cite[Proposition 3]{MR4711370} (see also \cite[Lemma A.1]{EZZL}) shows that for any $f\in (X_2)_+$ and all $\lambda> 0$ there is a constant $N>0$, independent of $f$, such that
\begin{equation*}
\Vert R(\lambda,\tilde{A}_2)f\Vert_{X_2}\ge \frac{1}{\lambda \bar{c}}N\Vert f\Vert_{X_1},
\end{equation*}
with $\bar{c}\coloneqq \sup_{1 \le j\le m}\Vert c_j\Vert_\infty$. The operator $G_2$ is clearly onto, and a direct computation shows that the Dirichlet operator associated with the pair $(A_2,G_2)$ is given by 
\begin{equation*}
(D_\la^2 \mathsf{a})(x)=\left(\frac{a_j}{c_j(\alpha_j\ell_j)}e^{-\la\int_{\alpha_j\ell_j}^{x}\frac{1}{c_j(s)}ds}\right)_{j=1}^m
\end{equation*}
for all $\la\in \C$, any $\mathsf{a}=(a_j)_{j=1}^m\in \R^m$, and a.e.\ $x\in [\alpha_j\ell_{j},\ell_{j}]$. It follows immediately that $D_\la^2$ is positive for all $\lambda>-\infty$. Furthermore, the operators $K_1$ and $K_2$ are positive as well.

Now, for $\lambda>\max\{s(\tilde{A}_1),s(\tilde{A}_2)\}$, define
\begin{equation*}
\A_\la\coloneqq K_1 D_\lambda^1 K_2 D_\lambda^2\in \R^{m\times m}.
\end{equation*}
Since $s(\tilde{A}_2)=-\infty$, it follows from Theorem \ref{Result1} that the system operator $\mathcal{A}$ as defined by \eqref{eqsysop123} generates a positive $C_0$-semigroup $\calT\coloneqq (\calT(t))_{t\in\R_+}$ on $X$ if and only if $r(\A_{\lambda_0})<1$ for some $\lambda_0>s(\tilde{A}_1)$. A straightforward computation yields, with \eqref{expsubd123},
\begin{equation*}
(\A_\la)_{i,j} =\frac{1}{\cosh(\gamma_i \alpha_i\ell_i)} \sum_{k=1}^m \frac{1}{\kappa_k \gamma_k \tanh(\gamma_k \alpha_k \ell_k)} \frac{c_j(\ell_j)}{c_j(\alpha_j \ell_j)} e^{-\la \int_{\alpha_j \ell_j}^{\ell_j} \frac{1}{c_j(s)} ds},\quad 1 \le i,j\le m.
\end{equation*}
Thus,
\begin{equation*}
\lim_{\la \to +\infty}\Vert \A_\lambda\Vert_{\R^{m\times m}}=0
\end{equation*}
and, hence, $\Vert \A_{\lambda_0}\Vert_{\R^{m\times m}}<1$ for some large enough $\lambda_0\in \R$. Therefore, the system operator $\mathcal{A}$ generates a positive $C_0$-semigroup $\calT\coloneqq (\calT(t))_{t\in\R_+}$ on $X$. In particular, the $\mathbf{aBCS}$ admits a unique positive mild solution $z\in \bm{C}(\R_+;X)$ for any $g,f\in X$.

Finally, for the statement on the exponential stability of solutions to the initial/boundary-value problem \eqref{transport-heat}-\eqref{transport-heat-IC}, we use Theorem \ref{Result2}. Recall that by the theorem exponential stability obtains if and only if the conditions $s(\tilde{A}_1),s(\tilde{A}_2)<0$ and $r(\A_0)<1$ hold. We have $s(\tilde{A}_2)=-\infty$, and by Lemma \ref{heat-star-exp-stab} we have $s(\tilde{A}_1)<0$. It remains to verify the remaining condition $r(\A_0)<1$; but we know that for $\lambda=0$, the condition \eqref{stability-cond} guarantees precisely this. The proof is complete.
\end{proof}

\section{Conclusions}\label{Sec:4}

In this paper, we have developed a new perturbation-based approach to the analysis of well-posedness and exponential stability of abstract boundary-coupled positive systems. We have exploited the positivity of both the internal dynamics of each subsystem and the boundary coupling operators. The methods stress the operator-theoretic aspects and rely on the properties of positive $C_0$-semigroups on Banach lattices which make it possible to obtain the existence, uniqueness, and regularity of solutions together with their exponential stability using general semigroup theory. In addition to the theoretical contributions, we have provided a number of applications, which demonstrate the applicability of the theory in obtaining well-posedness and exponential stability results.

In this present work, the well-posedness and exponential stability results are sharp in the sense that they cannot be relaxed under the assumptions on the coupled system. In this respect, our results are more directly applicable to the exponential stability problem for positive $C_0$-semigroups on Banach lattices. An obvious question, therefore, would ask about the stability of positive systems in the Banach lattice setting just in case the spectral criteria for exponential stability would not be satisfied. To this end, the analysis of a weaker form of stability, strong -- and possibly polynomial -- stability, could be a natural extension of the present study.

\bigskip\noindent
\textbf{Acknowledgments.} The research presented here was performed in part while the first author was visiting the Research Center for Complex Systems, Aalen University, Germany. He thanks Aalen University for its warm hospitality and the Stiftung KESSLER+CO für Bildung und Kultur for financial support under grant number EXPLOR-24MM.

\bibliographystyle{abbrvnat}
\bibliography{Bib}

\end{document}